\documentclass[12pt,reqno]{amsart}
\DeclareFontFamily{OML}{script}{}
\DeclareFontShape{OML}{script}{m}{it}
{ <5-20> rsfs10 }{}
\DeclareMathAlphabet{\mathscript}{OML}{script}{m}{it}

\renewcommand{\mathcal}[1]{{\mathscript #1}\hspace{0.2ex}}
\usepackage{cite}
\usepackage{color}
\ifx\red\undefined
\newcommand{\red}{\color{red}}

\fi
\usepackage[ansinew]{inputenc}
\usepackage[encapsulated]{CJK}

\usepackage{graphicx,graphics}
\usepackage{cite}
\usepackage{pifont}
\usepackage{amsthm}
\usepackage{subfigure}
\usepackage{amscd}
\usepackage{amsmath}
\usepackage{latexsym}
\usepackage{amsfonts}
\usepackage{amssymb}
\usepackage{color}
\usepackage{multicol}
\usepackage{amsmath,amssymb,amsthm,amsfonts,mathrsfs}
\usepackage{hyperref}
\hypersetup{hypertex=true,
            colorlinks=true,
            linkcolor=blue,
            anchorcolor=blue,
            citecolor=blue}
\usepackage{bm}

\usepackage{tikz}
\usepackage{pgfplots}
\usetikzlibrary{patterns}
\usepackage{enumitem}

\definecolor{ocre}{RGB}{64,123,121}

\allowdisplaybreaks[4]
\ifx\text\undefined
\newcommand{\text}{\mbox}
\fi
\ifx\operatorname\undefined
\newcommand{\operatorname}{\mathop}
\fi

\allowdisplaybreaks[3]

\newtheorem{theorem}{Theorem}[section]
\newtheorem{lemma}[theorem]{Lemma}

\newtheorem{remark}[theorem]{Remark}
\newtheorem{corollary}[theorem]{Corollary}

\theoremstyle{remark}

\numberwithin{equation}{section}

\begin{document}
\title{Liouville-type theorems and universal estimates for semilinear elliptic equations involving the product of the function and its gradient}
\author{Wenguo Liang and Zhengce Zhang}
\date{\today}
\address[Wenguo Liang]{School of Mathematics and Statistics, Xi'an Jiaotong University,
Xi'an, 710049, P. R. China}
\email{liangwenguo@stu.xjtu.edu.cn}
\address[Zhengce Zhang]{School of Mathematics and Statistics, Xi'an Jiaotong University,
Xi'an, 710049, P. R. China}
\email{zhangzc@mail.xjtu.edu.cn}
\thanks{Corresponding author: Zhengce Zhang}
\thanks{ }

\begin{abstract}
In this paper, we study local and global properties of positive solutions to the equation $-\Delta u=u^p|\nabla u|^q$ in a domain $\Omega$ of $\mathbb R^N$, where $p$ and $q$ are parameters. We introduce a linear operator to construct the differential inequality to obtain gradient estimates, and further establish Liouville-type theorems. As an application, we derive universal estimates for local solutions. Some of our results are new, as we extend the condition $p+q<(N+3)/(N-1)$ considered by He, Hu and Wang [Math. Z. 313 (2026), No. 6] to a wider range of parameters.
\end{abstract}

\maketitle

{\tableofcontents}  

\section{Introduction}

In this paper, we consider the qualitative properties of  equations with the form
\begin{equation}\label{eq1}
  -\Delta u=u^p|\nabla u|^q
\end{equation}
in $\Omega$, where $\Omega$ is a domain in $\mathbb R^N$, $N\geq 2$, and $p,q$ are parameters. For $A\geq 0$, let
\begin{equation}\label{levelset}
  \Omega_A:=\{x\in\Omega;\,|\nabla u(x)|=A\}
\end{equation}
be the level set of gradient for $u$. The term $u^p|\nabla u|^q$ is singular with $q<0$ in $\Omega_0$. A function $u\in C^1(\Omega)$ is said to be a weak solution of \eqref{eq1} provided that
\begin{equation*}
  u^p|\nabla u|^q\in L_{loc}^1(\Omega),
\end{equation*}
and
\begin{equation*}
  \int_\Omega \langle\nabla u,\nabla \psi\rangle =\int_\Omega u^p|\nabla u|^q\psi \quad\text{for all}\ \psi\in C^\infty_0(\Omega).
\end{equation*}
By the strong maximum principle \cite[Lemma 2.1]{Serrin-Zhou-Acta}, any nontrivial nonnegative solution of \eqref{eq1} is strictly positive.

In the case $q=0$, equation \eqref{eq1} reduces to the classical Lane--Emden equation
\begin{equation}\label{Lane-Emden}
  -\Delta u=u^p.
\end{equation}
By using an identity satisfied by the gradient of $u$ and delicate integral estimates, Gidas and Spruck \cite{Gidas-Spruck-1981-CPAM} first  proved that, for $1<p<p_S:=(N+2)/(N-2)$ and $N>2$, nonnegative solutions of \eqref{Lane-Emden} in $\mathbb R^N$ are identically zero. This method was later referred to as the integral Bernstein method in \cite{Bidaut-Veron-Duke-2019}. Chen and Li employed the moving plane method to provide another proof of this nonexistence result in \cite{Chen-Li-Duke-1991}.  The exponent $p_S$ is sharp in the sense that there exist positive radial solutions in $\mathbb R^N$ when $p\geq p_S$. In addition, asymptotic symmetry and local behavior of solutions in the critical case $p=p_S$ were obtained by Caffarelli, Gidas and Spruck \cite{Caffarelli-Giads-CPAM-1989}. Recently, by employing the Moser iteration technique, Wang and Wei \cite{Wang-Wei-JDE-2023} derived  local gradient estimates and  established the nonexistence of positive solutions to \eqref{Lane-Emden} on complete Riemannian manifolds, where the exponent $p$ is allowed to be negative.

In the case $p=0$ and $q>0$, equation \eqref{eq1} reduces to the Hamilton--Jacobi equation
\begin{equation}\label{Ham-Joc}
-\Delta u=|\nabla u|^q.
\end{equation}
By means of a direct Bernstein technique, Lions \cite{Lions-1985-JAM} established pointwise gradient estimates for solutions when $q>1$. Consequently, all classical  solutions of \eqref{Ham-Joc} in $\mathbb R^N$ are constants. Based on the Bernstein gradient estimates from \cite{Lions-1985-JAM} and moving plane method, Filippucci, Pucci and Souplet \cite{Filippucci-Pucci-Souplet-2020-CPDE} proved a Liouville-type classification for \eqref{Ham-Joc} in the half-space, showing that every classical solution is one-dimensional. For gradient estimates of equations concerning more general operators, we refer the reader to \cite{Attouchi-CVPDE-2020,Veron-2014-JFA,Chang-JDE-2023,Chang-DCDS-2020,Handong-Wang-2025-JDE,Han-He-Wang-JFA-2026,Tommaso-Porretta-2016-CPDE}.

We recall some relevant work concerning equation \eqref{eq1}, which was first introduced in its radial form by Caristi and Mitidieri in \cite{Caristi-AdcanceD-1997} to investigate the nonexistence of positive solutions. By using the pointwise Bernstein method and the integral Bernstein method, Bidaut-V\'{e}ron, Garc\'{\i}a-Huidobro and V\'{e}ron  \cite{Bidaut-Veron-Duke-2019} determined various regions of $(p,q)$ for which the Liouville-type property holds. Using a refined integral indentity and integral Bernstein method, Ma and Wu \cite{Ma-Wu-Bull-2026} improved the Liouville-type theorem, precisely, their result is optimal for  $0\leq q\leq1/(N-1)$. Recently, based on the Bernstein method, Lu \cite{Lu-arXive-2026} strengthened the sharp Liouville-type theorem for $0\leq q\leq1-\sqrt{1/(N-1)}$. For $q>2$, Filippucci, Pucci and Souplet \cite{Filippucci-Pucci-Souplet-Adv.stu-2020} obtained that all bounded classical solutions of \eqref{eq1} are constants. Subsequently, using Keller--Osserman comparison and bootstrap method, Bidaut-V\'{e}ron \cite{Bidaut-Veron-Adv.N-2021} completely removed the boundedness condition and generalized the Liouville-type theorem to quasilinear problem. For Liouville-type results concerning quasilinear and nonlocal elliptic equations, we refer to \cite{Chang-Hu-Zhang-NA-2022,Chang-Zhang-NA-2026, Filipuucci-NA-2009,Fillippucci-JDE-2011,Guo-Zhang-ProAMS-2025,Pokhozhaev-2001} and the references therein.

However, to the best of our knowledge, there seems much less known Liouville-type results for $q<0$, expect \cite{He-Hu-Wang-MZ-2026,Sun-Xiao-Xu-MathAnn-2022}. For $p+q<(N+3)/(N-1)$, by employing the Saloff--Coste Sobolev inequality and the Nash--Moser iteration technique, He, Wu and Wang \cite{He-Hu-Wang-MZ-2026} established a Liouville-type theorem for \eqref{eq1} on $N$-dimensional Riemannian manifold. Under some additional conditions on the volume growth of geodesic balls, Sun, Xiao and Xu \cite{Sun-Xiao-Xu-MathAnn-2022} proved that there exists no nontrivial non-negative solution to equation \eqref{eq1} on the Riemannian manifold for arbitrary $(p,q)\in \mathbb R^2$.

Beside the Liouville-type theorem for solutions, universal estimates of local solutions to equation \eqref{eq1} has received extensive attention.  Serrin and Zou \cite{Serrin-Zhou-Acta} observed that Liouville-type theorems for \eqref{Lane-Emden} can be regarded as a consequence and limiting case of universal boundedness theorems. Based on rescaling and a key ``doubling" property, Pol\'{a}\v{c}ik, Quittner and Souplet \cite{Polacik-Quittner-Souplet}  established the universal estimates of the form
\begin{equation}\label{form-M1}
  M_1(u(x)):=\left(u+|\nabla u|^{\beta_1}\right)(x)\leq C\left(1+{\rm dist}^{-\beta_2}(x,\partial\Omega)\right),\quad x\in\Omega,
\end{equation}
where $C,\beta_1,\beta_2>0$. $M_1$ is referred to as the blow-up quantity in the sense that, if \eqref{form-M1} fails, there exist sequence $\Omega_k$, $u_k$, and $x_k\in\Omega_k$ such that $M_1(u_k(x_k))\to\infty$ as $k\to\infty$. Furthermore, by rescaling $u_k$ yields a sequence of functions, and finally induces a subsequence that converges to a global solution of \eqref{Lane-Emden}, which contradicts with the known Liouville-type theorem.  However, since \eqref{eq1} admits arbitrary positive constant solutions, $M_1$ does not work in the rescaling procedure for \eqref{eq1} when $q> 0$. Consequently,  Baldelli and Filippucci \cite{Baldelli-Filippucci-2025-RIMU} obtained an alternative result to the priori estimates of positive solutions to \eqref{eq1}. Recently, for $q>0$, Lu and Zhu \cite{Lu-Zhu-JFA-2026} established a universal estimates of the form
\begin{equation}
  M_2(u(x)):=\left(|\nabla \ln u|^2+u^{p-1}|\nabla u|^q\right)(x)\leq C(N,p,q){\rm dist}^{-2}(x,\partial\Omega),\quad x\in\Omega.
\end{equation}

The primary objectives of the present paper are twofold. Firstly, we aim to establish Liouville-type theorems for \eqref{eq1} in a index region of parameters $p$ and $q$  which are wider than the previous one. Secondly, as applications, we intend to obtain a universal estimates of positive solutions to \eqref{eq1} when $q\geq 0$.

It turns out that the gradient estimates of solutions to equation \eqref{eq1} with $q<0$ is more complicated and thus more delicate analytical techniques are required. Specifically, on the one hand, the Keller--Osserman estimates \cite[Lemma 2.2]{Bidaut-Veron-Duke-2019} requires that equation \eqref{eq1} contains a superlinear absorption term, namely, $p+q>1$  and $q>0$. On the other hand, the Nash--Moser iteration technique employed in \cite{He-Hu-Wang-MZ-2026,Q-W-J-CVPDE-2026} relies on an inequality which holds for $p+q<(N+3)/(N-1)$. In our setting, by constructing a linear differential operator, together with introducing parameters, we establish differential inequalities on the superlevel set of gradient via choosing auxiliary functions, and further obtain local gradient estimates and Liouville-type theorems. Furthermore, instead of using the doubling lemma \cite[Lemma 5.1]{Polacik-Quittner-Souplet}, we derive the universal estimates by utilizing the function involving a blow-up quantity.

To better illustrate the main contributions and ideas of this paper, we shall elaborate on these methods in detail as follows.
\begin{enumerate}[
    label=(\roman*),
    leftmargin=*,
    labelsep=0.5em,        
]
\item \textit{Construction of a linear differential operator.} The Bernstein method depends on a differential inequality associated with $w$, where $w=|\nabla v|^2$, $v=f^{-1}(-u)$ and $f$ is a suitable auxiliary function. We observe that
\begin{equation*}
  \mathcal L(u^\alpha w^\gamma)=u^\alpha \mathcal L(w^\gamma)+w^\gamma\mathcal L(u^\alpha)-2\nabla u^\alpha\cdot\nabla w^\gamma,
\end{equation*}
where linear operator $\mathcal L:=-\Delta +\mathcal H\cdot\nabla$ and $\mathcal H\in \mathbb R^N$. It follows from \eqref{eq1} that
\begin{equation*}
  -\Delta u^\alpha=-\alpha(\alpha-1)u^{\alpha-2}|\nabla u|^2-\alpha u^{\alpha-1}\Delta u<0
\end{equation*}
for $|\nabla u|>0$ with $\alpha<0$, and thus the term $-w^\gamma \Delta u^\alpha$ arising from $w^\gamma\mathcal L(u^\alpha)$ contributes to deriving $\mathcal L(u^\alpha w^\gamma)<0$ for $w>0$. Nevertheless, we still need to handle the additional terms $w^\gamma\mathcal H\cdot\nabla u^\alpha$, $u^\alpha\mathcal H\cdot\nabla w^\gamma$ and $-2\nabla u^\alpha\cdot\nabla w^\gamma$. To this end, we rewrite $w^\gamma\mathcal H\cdot\nabla u^\alpha$ and $u^\alpha \mathcal H\cdot\nabla w^\gamma$ as $\mathcal H_\alpha\cdot \nabla (u^\alpha w^\gamma)$. Therefore, the operator $\mathcal L$ can be replaced by
\begin{equation*}
  \mathcal L_\alpha(z):=-\Delta z+\mathcal H_\alpha \cdot\nabla z
\end{equation*}
with
\begin{equation*}
 \mathcal H_\alpha= \left(q\frac{|f'|^q}{f'}(-f)^p w^{\frac{q-2}2}-2\frac{f''}{f'}+2\alpha \frac{f'}f\right)\nabla v.
\end{equation*}

\item \textit{The subsolution of $\mathcal L_\alpha(z)=0$ on the superlevel set of gradient.} The local pointwise gradient estimates relies on the inequality
    \begin{equation*}
      \mathcal L_\alpha(u^\alpha w^\gamma\eta)\leq -(u^\alpha w^\gamma\eta)^\theta+C(\eta),
    \end{equation*}
where $\eta$ is a cut-off function and $\theta,C(\eta)>0$. However, for given $x_0\in\Omega$, it seems that this inequality does not hold in $\{x\in B_{3R/4};\, (u^\alpha w^\gamma\eta)(x)>0\}$ for $\alpha<0$, where $B_R=B_R(x_0)$ is the open ball with center $x_0$ and radius $R={\rm dist}(x_0,\partial\Omega)$. Fortunately, it holds that \begin{equation*}
  \mathcal L_\alpha (u^\alpha w^\gamma\eta)<0 \quad \text{in}\ \{x\in B_{3R/4};\, m^{-\alpha}(u^\alpha w^\gamma\eta)(x)>C(\eta)\}
\end{equation*}
under the condition that $u$ admits a positive local lower bound $m$. By virtue of the maximum principle on the superlevel set of $m^{-\alpha}u^\alpha w^\gamma\eta$, we obtain the estimates for $u^\alpha w^\gamma$ in $\Omega$. We observe that an asymptotic behaviour \cite[Lemma 2.3]{Serrin-Zhou-Acta} of solutions to \eqref{eq1} at infinity guarantees the Liouville-type theorem. The idea of working on the superlevel set of $u^\alpha w^\gamma\eta$ is inspired by the argument proposed by Cirant and Goffi \cite{Girant-Goffi-ARMA-2021}; see also \cite{Cianchi-CPDE-2011,Grenon-CRMA-2006} and the references therein.
\end{enumerate}

It is noteworthy that by introducing parameters $\alpha$ and $\gamma$, we are able to propose gradient estimates over a wider range of exponents $p$ and $q$. Precisely, when $\alpha=0$, this estimates reduces to that of $w^\gamma$ with sufficiently large $\gamma$ under the condition $p+q<(N+3)/(N-1)$, whose critical line is parallel to line $p+q=1$. Furthermore, setting $\gamma=1$, the estimates reduces to $u^\alpha w$, in the range $p+(1-\alpha/2)q<(N+2)/N-\alpha/2$ for $\alpha<0$. Its critical line is a rotated version of line $p+q=1$. Figure \ref{fig:parameter_space} below illustrates distinct  separatrices for $N=2$ and $\alpha=-2\sigma=(1-\sqrt{(N+8)/N})/2$ on the plane.

In addition, during the derivation of a priori estimates for positive classical solutions, the blow-up quantities $M_1$ and $M_2$  fail to be used in rescaling procedure when proving the nontriviality and global boundedness of rescaled solutions. To overcome this difficulty, we introduce the quantity
\begin{equation*}
  M(u(x))=\left(u^p|\nabla u|^{q}\right)^{\beta}(x),\quad x\in\Omega
\end{equation*}
with $q\geq 0$ and $\beta=\beta(p,q)>0$. Drawing on the method proposed by Chen, Li, Wu and Xin \cite{Chen-Li-Wu-Xin-MathAnn-2026}, we build auxiliary functions containning $M$. This allows us to avoid invoking doubling lemma and reach the desired estimates via a direct calculation.

We are now in a position to present our main results.

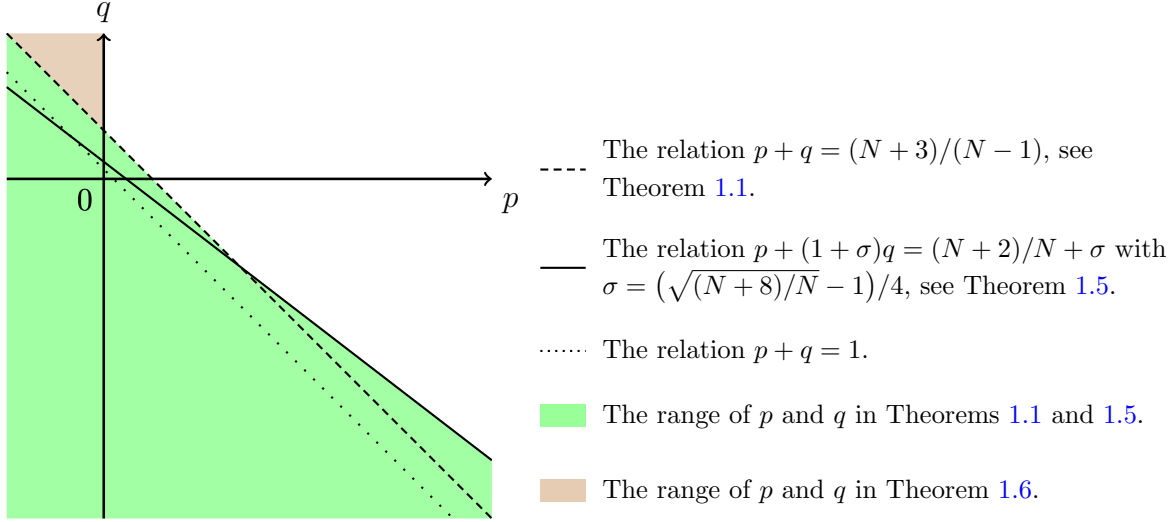
\begin{figure}
\centering
\begin{tikzpicture}
\begin{axis}[
    axis lines = middle,
    axis line style = {->,line width=0.98pt},
    axis equal,
    xlabel = {$p$},
    ylabel = {$q$},
    xmin = -10, xmax = 40,
    ymin = -35, ymax = 15,
    ticks = none,
    width = 8cm,
    height = 8cm,
    xlabel style={at={(axis description cs:1,0.65)},anchor=west}, 
    ylabel style={at={(axis description cs:0.20,1)},anchor=south}, 
    axis on top = true,  
]
\fill[green!40,opacity=0.9]
    (axis cs:-10,15)
   -- (axis cs:40, -35)
    -- (axis cs:-10, -35)
    -- cycle;

\fill[brown!40,opacity=0.9]
    (axis cs:-10,15)
   -- (axis cs:0, 15)
    -- (axis cs:0, 5)
    -- cycle;

\fill[green!40,opacity=0.9]
    (axis cs:40,-35)
   -- (axis cs:40, -29)
    -- (axis cs:14, -9)
    -- cycle;

\node at (axis cs:0,0) [below left] {$0$};

\addplot[
    domain=-10:40,
    samples=200,
    line width=0.8pt,
   loosely dotted
] (x,{1-x});

\addplot[
    domain=-10:40,
    samples=200,
    line width=0.8pt,
    solid
] (x,{(2.3-x)/1.3});

\addplot[
    domain=-10:40,
    samples=200,
    line width=0.8pt,
    densely dashed
] (x,{5-x});

\node at (axis cs:0,0) [below left] {$0$};

\end{axis}
\end{tikzpicture}
\begin{tikzpicture}

\draw[densely dashed, thick]
  (3,4.5)-- (3.6,4.5);
\node[anchor=west] at (3.5,4.5)
{\begin{tabular}{l}
\footnotesize The relation $p+q=(N+3)/(N-1)$, see\\
\footnotesize Theorem \ref{them:p+q<(N+3)/}.
    \end{tabular}
};

\draw[solid,thick]
  (3,3.2)-- (3.6,3.2);
\node[anchor=west] at (3.5,3.2)
{\begin{tabular}{l}
\footnotesize The relation $p+(1+\sigma)q=(N+2)/N+\sigma$ with \\
\footnotesize $\sigma=\big(\sqrt{(N+8)/N}-1\big)/4$, see Theorem \ref{corol:<p+q<}.
    \end{tabular}
};

\draw[dotted,thick]
  (3,2.1)-- (3.6,2.1);
\node[anchor=west] at (3.5,2.1)
{\begin{tabular}{l}
\footnotesize The relation $p+q=1$.
    \end{tabular}
};

\path[fill=green!40]
  (3,1.1)-- (3.6,1.1)-- (3.6,1.4)-- (3,1.4)-- cycle;
\node[anchor=west] at (3.5,1.25)
{\begin{tabular}{l}
\footnotesize The range of $p$ and $q$ in Theorems \ref{them:p+q<(N+3)/} and \ref{corol:<p+q<}.
    \end{tabular}
};

\path[fill=brown!40]
  (3,0.1)-- (3.6,0.1)-- (3.6,0.4)-- (3,0.4)-- cycle;
\node[anchor=west] at (3.5,0.25)
{\begin{tabular}{l}
\footnotesize The range of $p$ and $q$ in Theorem \ref{them:p<=0}.
    \end{tabular}
};

\end{tikzpicture}	
\caption{The range of $p$ and $q$ for Liouville results when $N=2$.}\label{fig:parameter_space}
\end{figure}

\begin{theorem}\label{them:p+q<(N+3)/}
 Let $u$ be a positive solution of \eqref{eq1} in $\Omega$.  Assume that
\begin{equation*}
p+q<\frac{N+3}{N-1}.
\end{equation*}
Then there exists a constant $C=C(N,p,q)>0$ such that
\begin{equation}\label{estm-nabla u/u}
  \left(\frac{|\nabla u|}{u}\right)(x)\leq C {\rm dist}^{-1}(x,\partial \Omega),\quad x\in\Omega.
\end{equation}
\end{theorem}

We remark that gradient estimates holds for the same range of $p,q$ values as those arguments in \cite[Corollary 1.3]{He-Hu-Wang-MZ-2026}, whose proof employs integral inequalities and iteration technique. By contrast, our arguments rely on local differential inequalities and the maximum principle,  we apply the transformation $v=f^{-1}(-u)$ with $f(s)=-e^s$ and define $w=|\nabla v|^2$. Then for $\gamma\geq 1$,  we arrive at
\begin{equation*}
  w^{1-\gamma}\mathcal L(w^\gamma)\leq \gamma\left(-(\gamma-1)w^{-1}|\nabla w|^2-2|D^2 v|^2+2(p+q-1)u^{p+q-1} w^{\frac{q+2}2}\right),
\end{equation*}
where $\mathcal L(z):= -\Delta z-\mathcal H\cdot\nabla z$ with $\mathcal H:=\mathcal H(v,\nabla v,f)$. Introduction of the parameter $\gamma$ allows the positive term $u^{p+q-1}w^{(q+2)/2}$ can be absorbed by the negative terms $-(\gamma-1)w^{-1}|\nabla w|^2$ and $-|D^2v|^2$.  Upon choosing $\gamma$ sufficiently large, we get
\begin{equation*}
\mathcal L(w^\gamma)\leq -u^{2(p+q-1)}w^{q+\gamma-1}-w^{1+\gamma}.
\end{equation*}
Once this is carried out, replacing $\mathcal L(w^\gamma)$ by $\mathcal L(w^\gamma \eta)$ ensures that the preceding argument remains valid for a given cut-off function $\eta$. In addition, it holds that
\begin{equation*}
  \mathcal L(w^\gamma \eta)<0\quad \text{in}\ \Omega':=\big\{x\in B_{3R/4};\,(w^\gamma\eta)(x)>CR^{-2\gamma}\big\},
\end{equation*}
in which $x_0\in\Omega$, $R={\rm dist}(x_0,\partial \Omega)>0$ and $C>0$ depends on parameters. Applying the maximum principle on $\Omega'$, we deduce the gradient estimates in $\Omega$.

From the gradient estimates \eqref{estm-nabla u/u}, we present the following Liouville-type theorem.
\begin{corollary}\label{corol:p+q<}
Assume that $p+q<(N+3)/(N-1)$. Then \eqref{eq1} possesses no nontrivial  positive solution in $\mathbb R^N$.
\end{corollary}

\begin{remark}
By applying the gradient estimates \eqref{estm-nabla u/u}, we conclude that $\nabla u\equiv0$, and $u$ is a positive constant. However, for $q<0$, this contradicts  with  $u^p|\nabla u|^q\in L^1_{loc}(\mathbb R^N)$.
\end{remark}

The next result provides the gradient estimates for positive solutions under a local lower bound.

\begin{theorem}\label{them:p+q>=(N+3)/}
Let $u$ be a positive solution of \eqref{eq1} in $\Omega$ that satisfies $u\geq m$ for some constant $m>0$. Assume that
\begin{equation}\label{Q(p,q)<sigma}
 p+\left(1+\sigma\right)q<\frac{N+2}N+\sigma
\end{equation}
with $0<\sigma<\big(\sqrt{(N+8)/N}-1\big)/4$. Then there exists a constant $C=C(N,p,q,\sigma)>0$ such that
\begin{equation}\label{|nabla-u|<m^-1}
  \left|\nabla u^{-\sigma}(x)\right|\leq Cm^{-\sigma}{\rm dist}^{-1}(x,\partial\Omega),\quad x\in\Omega.
\end{equation}
\end{theorem}

The gradient estimates in range \eqref{Q(p,q)<sigma} relies on the differential inequality for $u^\alpha w$ with $\alpha=-2\sigma<0$. Upon choosing a proper auxiliary function $f$, the negative perturbation term $-w\Delta u^\alpha$ allows us to obtain the differential inequality in a wider range. Indeed, it holds
\begin{equation*}
  u^{-\alpha}\mathcal L_\alpha(u^\alpha w)\leq \left(p+\big(1-\frac\alpha2\big)q-\frac{N+2}N+\frac{\alpha}2\right)u^{p+q-1}w^{\frac{q+2}2}-C(N,\alpha)w^2,
\end{equation*}
with $C(N,\alpha)>0$, where $\mathcal L_\alpha:=-\Delta-\mathcal H_\alpha\cdot\nabla $ and $\mathcal H_\alpha:=\mathcal H(\alpha,v,\nabla v,f)$. We remark that condition \eqref{Q(p,q)<sigma}  together with the  local lower bound hypothesis on  $u$ implies that
\begin{equation*}
  \mathcal L_\alpha(u^\alpha w)\leq -C(N,\alpha)u^\alpha w^2\leq -C(N,\alpha)m^{-\alpha} (u^{\alpha}w)^2.
\end{equation*}
Replacing $\mathcal L_\alpha(u^\alpha w)$ with $\mathcal L_\alpha(u^\alpha w\eta)$, we verify that the previous arguments are valid for the cut-off function $\eta$, so the inequality can be expressed as
\begin{equation*}
  \mathcal L_\alpha(u^\alpha w\eta)<0\quad \text{in} \ \big\{x\in B_{3R/4};\,(u^\alpha w\eta)(x)>Cm^{\alpha}R^{-2}\big\}.
\end{equation*}
The estimates for $u^\alpha w$ now follows from the maximum principle.

As observed in Serrin and Zou \cite[Lemma 2.3]{Serrin-Zhou-Acta}, the lower bound estimates
\begin{equation*}
  u(x)\geq C|x|^{-(N-2)},\quad |x|>1,
\end{equation*}
for positive weak super-harmonic functions are proved, where $C$ depends on $N$ and $\min_{|x|=2}u$. Based upon the asymptotic estimates at infinity, we then establish our Liouville-type theorem.

\begin{theorem}\label{corol:<p+q<}
Assume that $N$, $p$, $q$, and $\sigma$ satisfy
\begin{equation*}
 p+\left(1+\sigma\right)q<\frac{N+2}N+\sigma
\end{equation*}
with $0<\sigma<\big(\sqrt{(N+8)/N}-1\big)/4$. Then \eqref{eq1} possesses no nontrivial positive solution in $\mathbb R^N$.
\end{theorem}

The following result was established in \cite[Theorem 1.5]{Bidaut-Veron-Adv.N-2021} by means of the transformation $u=v^b$ for $b\geq1$ and a Keller--Osserman type inequality. For the sake of completeness, we present an alternative proof, whose techniques will be adopted throughout this paper.

\begin{theorem}\label{them:p<=0}
Let $u$ be a positive solution of \eqref{eq1} in $\Omega$. Assume that $p\leq 0$ and
\begin{equation*}
  p+q\geq \frac{N+3}{N-1}.
\end{equation*}
Then there exists a constant $C=C(N,p,q)>0$ such that
\begin{equation}\label{estm-nabla u}
  \left|\nabla u^{\frac{p+q-1}{q-1}}(x)\right|\leq C {\rm dist}^{-\frac{1}{q-1}}(x,\partial \Omega),\quad x\in\Omega.
\end{equation}
If $\Omega=\mathbb R^N$, then $u$ is a constant.
\end{theorem}

Turning to the second principal goal of this paper, we derive local pointwise estimates for positive solution to \eqref{eq1}.
Let $\mathbb R^2_+=(-\infty,\infty)\times[0,\infty)$. We define the index region
\begin{align*}
 \mathcal{R}_L= \left\{(p, q) \in \mathbb{R}_+^2 \middle|
 \begin{array}{l}
\text{There exists no nontrivial positive } \\
\text{solution
  of \eqref{eq1} in $\mathbb R^N$ with indices}\ (p,q).
\end{array} \right\},
\end{align*}
for which the Liouville-type theorem for equation \eqref{eq1} is valid. We next present the following universal priori estimates.

\begin{theorem}\label{them:uni-u}
Let $\Omega$ be an arbitrary domain of $\mathbb R^N$. Assume that $u\in C^2(\Omega)$ is a positive solution of \eqref{eq1} in $\Omega$ with $(p,q)\in\mathcal R_L$, and
\begin{equation}\label{beta/p>0}
  \frac{p+q-1}{2p+q}>0.
\end{equation}
Then there exists $C=C(N,p,q)>0$ such that
\begin{equation}\label{uniform-u}
(u^p|\nabla u|^{q})^{\frac{p+q-1}{2p+q}}(x)\leq C{\rm dist}^{-1}(x,\partial\Omega),\quad x\in\Omega.
\end{equation}
\end{theorem}

\begin{remark}
{\rm{(i)}} Theorem \ref{them:uni-u} covers the estimates in terms of  Lane-Emden equation \eqref{Lane-Emden} and Hamilton-Jacobi equation \eqref{Ham-Joc}. Indeed, when $q=0$, in view of the Liouville-type results established in \cite{Gidas-Spruck-1981-CPAM,Wang-Wei-JDE-2023} for $p<p_S:=(N+2)/(N-2)$ with $N>2$,  the estimates \eqref{uniform-u} reduces to
\begin{equation}\label{u<cd^-1}
  u^{\frac{p-1}2}(x)\leq C{\rm dist}^{-1}(x,\partial\Omega),\quad x\in\Omega
\end{equation}
with $1<p<p_S$ or $p<0$. For $1<p<p_S$, \eqref{u<cd^-1} is the well-known universal boundedness result established by Dancer \cite[Lemma 1]{Dancer-MathZ-1998}. However, when $p<0$, \eqref{uniform-u} yields a lower bound estimates
\begin{equation*}
  u(x)\geq  C{\rm dist}^{-\frac{2}{p-1}}(x,\partial\Omega),\quad x\in\Omega.
\end{equation*}
For $p=0$, there holds
\begin{equation*}
  |\nabla u(x)|\leq C{\rm dist}^{-\frac1{q-1}}(x,\partial\Omega),\quad x\in\Omega
\end{equation*}
when $q>1$, which is proved by Lions \cite[Theorem IV.1]{Lions-1985-JAM}.

{\rm (ii)} Based on the definition of $\mathcal R_L$,  estimates \eqref{uniform-u} is also valid for indices considered in \cite{Bidaut-Veron-Adv.N-2021,Bidaut-Veron-Duke-2019,Filippucci-Pucci-Souplet-Adv.stu-2020,Lu-arXive-2026,Ma-Wu-Bull-2026} under assumption \eqref{beta/p>0}, where Liouville-type theorems have been established for $p,q\geq 0$.
\end{remark}

This paper is organized as follows. In Section \ref{sect:diff-ineq}, we establish  differential inequalities involving the solution and its gradient of equation \eqref{eq1}, and  then present the maximum principle. In Section \ref{sect:gradient-est}, by choosing different auxiliary functions, we provide pointwise gradient estimates as well as the corresponding Liouville-type theorems.  In Section \ref{sect:univer-esti}, we combine the Liouville-type theorems with the rescaling method to derive universal estimates.

\section{Elliptic differential operators and inequalities}\label{sect:diff-ineq}

We begin by introducing some additional notations. Let $a\in (0,1)$ be a constant to be chosen later, and fix $x_0\in\Omega$.  Define $R={\rm dist}(x_0,\partial\Omega)$ and $R'=3R/4$. We then choose a cut-off function $\eta\in C^\infty(\overline B_R)$, $0\leq \eta\leq 1$ such that $\eta=1$ in $|x-x_0|\leq R/2$, $\eta=0$ in $|x-x_0|\geq R'$, and such that
\begin{equation}\label{dfi-eta}
\left.\begin{array}{rl}
|\nabla \eta| & \leq C R^{-1} \eta^a \\
\left|D^2 \eta\right|+\eta^{-1}|\nabla \eta|^2 & \leq C R^{-2} \eta^a
\end{array}\right\}\quad \text { for }\left|x-x_0\right|<R^{\prime},
\end{equation}
where $C=C(a)>0$ and $B_R=B_R(x_0)$. Indeed, such a function is given in \cite{Souplet-Zhang} as $\eta=\rho^k$, where $\rho(x)=1-R'^{-2}|x-x_0|^2$, $x\in B_R$ and $k\geq2/(1-a)$.

For the remainder of this section,  we establish differential inequalities that contain undetermined auxiliary functions for solutions to \eqref{eq1}. These inequalities will be frequently  applied to derive the gradient estimates. Let $f$ denote a function to be specified later, such that
\begin{equation}\label{asum-f}
f\in C^3(\mathbb R)\quad \text{and}\quad f'\neq 0.
\end{equation}
We set
\begin{equation}\label{u=-f(v)}
  v=f^{-1}(-u),\quad w=|\nabla v|^2.
\end{equation}
By a direct computation, we see that $v$ satisfies the equation
\begin{equation}\label{eq-v}
  -\Delta v=-\frac{|f'|^{q}}{f'}(-f)^p w^{\frac q2}+\frac{f''}{f'} w.
\end{equation}
For convenience, we omit  the variable $v$ of $f$, $f'$ and $f''$ throughout. Elliptic regularity theory ensures that the solution $u$ is smooth outside the set $\Omega_0$ given in \eqref{levelset}, and hence we can differentiate the equation. For $w=|f'|^{-2}|\nabla u|^2>0$, we employ Bochner's identity
\begin{equation*}
  \Delta w=2\langle \nabla \Delta v,\nabla v\rangle+2|D^2 v|^2
\end{equation*}
to obtain
\begin{align}\label{eq-w}\nonumber
  &-\Delta w+q\frac{|f'|^q}{f'}(-f)^p w^{\frac{q-2}{2}}\langle\nabla v,\nabla w\rangle-2\frac{f''}{f'}\langle\nabla v,\nabla w\rangle\\
  =&\Big[-2(q-1)(-f)^p|f'|^{q-2}f''+2p(-f)^{p-1}|f'|^q\Big]w^{\frac{q+2}2}+2\left(\frac{f''}{f'}\right)'w^2-2|D^2v|^2,
\end{align}
where $|D^2 v|^2=\sum_{i,j=1}^N(v_{ij})^2$. We consider the operator $\mathcal L_\alpha$ defined as
\begin{equation}\label{def-L-alpha}
  \mathcal L_\alpha(z)=-\Delta z+\mathcal H_\alpha\cdot \nabla z,
\end{equation}
where
\begin{equation}\label{H_alpha}
  \mathcal H_\alpha:= \left(q\frac{|f'|^q}{f'}(-f)^p w^{\frac{q-2}2}-2\frac{f''}{f'}+2\alpha \frac{f'}f\right)\nabla v.
\end{equation}
Setting $\alpha =0$ in \eqref{def-L-alpha}, then \eqref{eq-w} can be written as
\begin{equation}\label{L(w)}
  \mathcal L(w):=\mathcal L_0(w)=-2|D^2 v|^2+\mathcal N(w),
\end{equation}
where
\begin{equation*}
  \mathcal N(w):=\Big[-2(q-1)(-f)^p|f'|^{q-2}f''+2p(-f)^{p-1}|f'|^q\Big]w^{\frac{q+2}2}+2\left(\frac{f''}{f'}\right)'w^2.
\end{equation*}

We next establish a general local differential inequality for $u^\alpha w^\gamma$ via a cut-off function.

\begin{lemma}\label{lem:L(u-alpah-w-gamma)}
  Assume that $u$ is a positive solution of \eqref{eq1} in $\Omega$. Then for any $\alpha\in \mathbb R$ and $\gamma \geq 1$,
  \begin{align}\label{L_a(u^aw^ga)}\nonumber
   &u^{1-\alpha}w^{1-\gamma}\mathcal L_\alpha(u^\alpha w^\gamma\eta)\\\nonumber
\leq &\gamma u\Big[-(\gamma-1)w^{-1}|\nabla w|^2-\Big(2(q-1)u^p|f'|^{q-2}f''-\left(2p-\gamma^{-1}\alpha (q-1)\right)u^{p-1}|f'|^q\Big)w^{\frac{q+2}2}\\
  &+2\left(\frac{f''}{f'}\right)'w^2-2|D^2v|^2\Big]\eta+\alpha \left((\alpha+1)\frac{(f')^2}{u}+2f''\right)w^2\eta-uw\Delta\eta\\\nonumber
  &+uw\left(|q|u^{p}|f'|^{q-1}w^{\frac{q-1}2}+2\left|\frac{f''}{f'}\right|w^{\frac12}\right)|\nabla\eta|-2\gamma u\langle\nabla w,\nabla \eta\rangle
  \end{align}
holds in $\{x\in B_{R'};\, |\nabla u(x)|>0\}$.
\end{lemma}

\begin{proof}
From \eqref{L(w)}, we have
\begin{align}\label{Del-w-gam}\nonumber
  -\Delta(w^\gamma)=&-\gamma(\gamma-1)w^{\gamma-2}|\nabla w|^2-\gamma w^{\gamma-1}\Delta w\\
  =&-\gamma(\gamma-1)w^{\gamma-2}|\nabla w|^2-2\gamma w^{\gamma-1}|D^2 v|^2\\\nonumber
  &+\gamma w^{\gamma-1}\mathcal N(w)-\left(q\frac{|f'|^q}{f'}(-f)^p w^{\frac{q-2}{2}}-2\frac{f''}{f'}\right)\langle\nabla v,\nabla w^\gamma\rangle.
\end{align}
By a direct computation, it follows from \eqref{u=-f(v)} that
\begin{align}\label{Del-u-alp}\nonumber
  -\Delta(u^\alpha)=&-\alpha(\alpha-1)u^{\alpha-2}|\nabla u|^2-\alpha u^{\alpha-1}\Delta u\\
  =&-\alpha(\alpha-1)u^{\alpha-2}(f')^2w+\alpha u^{\alpha-1}(-f)^p|f'|^qw^{\frac{q}2}.
\end{align}
As for $u^\alpha w^\gamma \eta$, a calculation yields that
\begin{align}\label{Del-uweta}\nonumber
  &-\Delta\left(u^\alpha w^\gamma\eta\right)\\
  =&-u^\alpha \Delta\left(w^\gamma\right)\eta-w^\gamma\Delta\left(u^\alpha\right)\eta-u^\alpha w^\gamma\Delta\eta-2\langle\nabla u^\alpha,\nabla w^\gamma\rangle\eta-2\langle\nabla\left(u^\alpha w^\gamma\right),\nabla \eta\rangle.
\end{align}
Substituting \eqref{Del-w-gam} and \eqref{Del-u-alp} into \eqref{Del-uweta}, and recalling that $u=-f(v)$, we obtain
\begin{align}\label{Del-u-a-w-gam}\nonumber
  &-u^{1-\alpha} w^{1-\gamma}\Delta (u^\alpha w^\gamma\eta)\\\nonumber
  =&\gamma u\Big[-(\gamma-1)w^{-1}|\nabla w|^2-2|D^2v|^2-\Big(2(q-1)u^p|f'|^{q-2}f''-2pu^{p-1}|f'|^q\Big)w^{\frac{q+2}2}\\
  &+2\left(\frac{f''}{f'}\right)'w^2\Big]\eta+\alpha w\Big(-(\alpha-1)u^{-1}(f')^2w+u^p|f'|^q w^{\frac{q}2}\Big)\eta\\\nonumber
  &-uw^{1-\gamma}\eta\left(q\frac{|f'|^q}{f'}u^p w^{\frac{q-2}{2}}-2\frac{f''}{f'}\right)\langle\nabla w^\gamma,\nabla v\rangle-u w\Delta\eta\\\nonumber
  &-2u^{1-\alpha}w^{1-\gamma}\eta\langle\nabla u^{\alpha},\nabla w^{\gamma}\rangle-2 u^{1-\alpha}w^{1-\gamma}\langle\nabla(u^\alpha w^\gamma),\nabla\eta\rangle.
\end{align}

Next, due to the definition of $\mathcal H_\alpha$, we have
\begin{align}\label{H-alpha-expres}\nonumber
  &u^{1-\alpha}w^{1-\gamma}\mathcal H_\alpha\cdot \nabla (u^\alpha w^\gamma \eta)\\
  =& u^{1-\alpha}w^{1-\gamma}\left(q\frac{|f'|^q}{f'}(-f)^p w^{\frac{q-2}2}-2\frac{f''}{f'}+2\alpha \frac{f'}f\right)\nabla v\cdot\nabla(u^\alpha w^\gamma\eta)\\\nonumber
  =&\left(q\frac{|f'|^q}{f'}(-f)^p w^{\frac{q-2}2}-2\frac{f''}{f'}\right)\Big[u^{1-\alpha}w\eta\langle\nabla v,\nabla u^\alpha\rangle+w^{1-\gamma}u\eta\langle\nabla v,\nabla w^\gamma\rangle+uw\langle\nabla v,\nabla \eta\rangle\Big]\\\nonumber
  &+2\alpha \frac{f'}f u^{1-\alpha} w^{1-\gamma}\langle\nabla v,\nabla(u^\alpha w^\gamma\eta)\rangle.
\end{align}
Noting that
\begin{align*}
  \langle\nabla v,\nabla u^\alpha\rangle=\alpha u^{\alpha-1}\langle\nabla v,\nabla u\rangle=-\alpha u^{\alpha-1}f'w,
\end{align*}
we have
\begin{equation}\label{nabla_ualpha1}
  u^{1-\alpha}w\eta\left(q\frac{|f'|^q}{f'}u^p w^{\frac{q-2}{2}}-2\frac{f''}{f'}\right)\langle\nabla v,\nabla u^\alpha\rangle=-\alpha qu^p|f'|^qw^{\frac{q+2}2}\eta+2\alpha f''w^2\eta.
\end{equation}
To estimate $2\alpha f'f^{-1}\langle\nabla v,\nabla(u^\alpha w^\gamma\eta)\rangle$, we make a direct computation to obtain
\begin{align*}
2\langle\nabla u^\alpha,\nabla w^\gamma\rangle\eta=&2\alpha\frac{f'}{f}\langle\nabla v,u^\alpha\eta\nabla w^\gamma\rangle\\
=&2\alpha\frac{f'}f\big\langle\nabla v,\nabla(u^\alpha w^\gamma\eta)-w^\gamma\eta\nabla u^\alpha-u^\alpha w^\gamma\nabla\eta\big\rangle\\\nonumber
=&2\alpha\frac{f'}f\langle\nabla v,\nabla(u^\alpha w^\gamma\eta)\rangle+2\alpha^2\frac{(f')^2}{f}u^{\alpha-1}w^{\gamma+1}\eta-2\alpha\frac{f'}fu^\alpha w^\gamma\langle\nabla v,\nabla\eta\rangle,
\end{align*}
and then we get
\begin{align}\label{**2alp-ff}\nonumber
  &2\alpha\frac{f'}fu^{1-\alpha}w^{1-\gamma}\langle\nabla v,\nabla(u^\alpha w^\gamma\eta)\rangle\\
  =&2u^{1-\alpha}w^{1-\gamma}\langle\nabla u^\alpha,\nabla w^\gamma\rangle\eta-2\alpha^2\frac{(f')^2}{f}w^{2}\eta+2\alpha\frac{f'}fu w\langle\nabla v,\nabla\eta\rangle.
\end{align}
Substituting \eqref{nabla_ualpha1} and \eqref{**2alp-ff} into \eqref{H-alpha-expres}, it leads to
\begin{align}\label{uwH-alpha}\nonumber
 &u^{1-\alpha}w^{1-\gamma}\mathcal H_\alpha\cdot \nabla (u^\alpha w^\gamma \eta)\\
 =&-\alpha qu^p|f'|^qw^{\frac{q+2}2}\eta+2\alpha f''w^2\eta+2u^{1-\alpha}w^{1-\gamma}\langle\nabla u^\alpha,\nabla w^\gamma\rangle\eta-2\alpha^2\frac{(f')^2}{f}w^{2}\eta\\\nonumber
 &+2\alpha\frac{f'}fu w\langle\nabla v,\nabla\eta\rangle+\left(q\frac{|f'|^q}{f'}(-f)^p w^{\frac{q-2}2}-2\frac{f''}{f'}\right)\Big[w^{1-\gamma}u\eta\langle\nabla v,\nabla w^\gamma\rangle+uw\langle\nabla v,\nabla \eta\rangle\Big]
\end{align}

It follows from \eqref{def-L-alpha}, \eqref{Del-u-a-w-gam} and \eqref{uwH-alpha} that
\begin{align}\label{L-u-alpha-2}\nonumber
&u^{1-\alpha}w^{1-\gamma}\mathcal L_\alpha(u^\alpha w^\gamma\eta)\\\nonumber
=&\gamma u\Big[-(\gamma-1)w^{-1}|\nabla w|^2-\Big(2(q-1)u^p|f'|^{q-2}f''-\left(2p-\gamma^{-1}\alpha (q-1)\right)u^{p-1}|f'|^q\Big)w^{\frac{q+2}2}\\
  &+2\left(\frac{f''}{f'}\right)'w^2-2|D^2v|^2\Big]\eta+\alpha \left((\alpha+1)\frac{(f')^2}{u}+2f''\right)w^2\eta-uw\Delta\eta\\\nonumber
  &-2u^{1-\alpha}w^{1-\gamma}\langle\nabla(u^\alpha w^\gamma),\nabla\eta\rangle+2\alpha\frac{f'}fu w\langle\nabla v,\nabla\eta\rangle+uw\Big(q\frac{|f'|^q}{f'}u^pw^{\frac{q-2}2}-2\frac{f''}{f'}\Big)\langle\nabla v,\nabla\eta\rangle.
\end{align}
Moreover, it has
\begin{align}\label{nabla-v,nabla-eta}\nonumber
-2\langle\nabla(u^\alpha w^\gamma),\nabla\eta\rangle=&-2w^\gamma\langle\nabla u^\alpha,\nabla\eta\rangle-2u^\alpha \langle\nabla w^\gamma,\nabla\eta\rangle\\
=&2\alpha u^{\alpha-1}w^\gamma f'\langle\nabla v,\nabla\eta\rangle-2\gamma u^\alpha w^{\gamma-1}\langle\nabla w,\nabla\eta\rangle\\\nonumber
=&-2\alpha \frac{f'}{f}u^\alpha w^\gamma\langle\nabla v,\nabla\eta\rangle-2\gamma u^\alpha w^{\gamma-1}\langle\nabla w,\nabla\eta\rangle.
\end{align}
For those terms involving $\nabla \eta$, we apply Cauchy--Schwarz's inequality to obtain
\begin{equation}\label{nabla-eta-1}
\left(q\frac{|f'|^q}{f'}u^pw^{\frac{q-2}2}-2\frac{f''}{f'}\right)\langle\nabla v,\nabla\eta\rangle\leq \left(|q|u^{p}|f'|^{q-1}w^{\frac{q-1}2}+2\left|\frac{f''}{f'}\right|w^{\frac12}\right)|\nabla\eta|.
\end{equation}
Substituting \eqref{nabla-v,nabla-eta} and \eqref{nabla-eta-1} into \eqref{L-u-alpha-2}, we arrive at
\begin{align*}\label{L_a(u^aw^ga)}\nonumber
   &u^{1-\alpha}w^{1-\gamma}\mathcal L_\alpha(u^\alpha w^\gamma\eta)\\\nonumber
\leq &\gamma u\Big[-(\gamma-1)w^{-1}|\nabla w|^2-\Big(2(q-1)u^p|f'|^{q-2}f''-\left(2p-\gamma^{-1}\alpha (q-1)\right)u^{p-1}|f'|^q\Big)w^{\frac{q+2}2}\\
  &+2\left(\frac{f''}{f'}\right)'w^2-2|D^2v|^2\Big]\eta+\alpha \left((\alpha+1)\frac{(f')^2}{u}+2f''\right)w^2\eta-uw\Delta\eta\\\nonumber
  &+uw\left(|q|u^{p}|f'|^{q-1}w^{\frac{q-1}2}+2\left|\frac{f''}{f'}\right|w^{\frac12}\right)|\nabla\eta|-2\gamma u\langle\nabla w,\nabla \eta\rangle.
  \end{align*}
Therefore, we finish the proof.
\end{proof}

We further require the following weak form of maximum principle, which allows the coefficients of the first-order derivative to be unbounded, but bounded on a subset of $\Omega$.

\begin{lemma}\label{lem:maxi-pric}
Let $\Omega\subset\mathbb R^N$ be a domain, $\vec{b}\in \mathbb R^N$ and $A\geq0$. Set
\begin{equation*}
  \Omega_A^+=\{x\in\Omega;\ z(x)>A\}.
\end{equation*}
Assume that $z$ is positive continuous in $\Omega$ and $C^2$ on $\Omega_A^+$. If $z$ satisfies
\begin{equation}\label{z>A}
  -\Delta z+\vec{b}\cdot \nabla z<0\quad \text{in}\ \Omega_A^+,
\end{equation}
and
\begin{equation*}
\vec{b}\in L^{\infty}(\Omega_A^+),
\end{equation*}
then $z\leq A$ in $\Omega$.
\end{lemma}

\begin{proof}
Suppose for contradiction that there exists $x_0\in \Omega$ such that $z(x_0)>A$. Set
\begin{equation*}
  \bar z(x)=z(x)-A>0,\quad x\in\Omega_A^+,
\end{equation*}
then
\begin{equation*}
  -\Delta\bar z+\vec{b}\cdot\nabla \bar z<0 \quad\text{in} \ \Omega_A^+.
\end{equation*}
However, $\bar z=0$ on $\partial \Omega_A^+$. The maximum principle of classical solutions deduces that $\bar z\leq 0$ in $\Omega_A^+$, which contradicts
\begin{equation*}
0<z(x_0)-A=\bar z(x_0)\leq 0.
\end{equation*}
It follows that $\Omega_A^+=\emptyset$, so $z\leq A$ in $\Omega$, which finishes the proof.
\end{proof}

\section{Gradient estimates}\label{sect:gradient-est}

In this section, we establish the local and global pointwise gradient estimates for solutions to equation \eqref{eq1}. As in past studies on gradient estimates, the proof is based on the elaborate Bernstein technique. Within this framework, by selecting different parameters $\alpha$ and $\gamma$ for $u^\alpha w^\gamma$, we demonstrate the differential inequalities of $(p,q)$ over various ranges.

For pointwise gradient estimates, these points in which $\nabla u=0$ are trivial. It therefore suffices to establish the estimates in $\Omega\setminus\Omega_0$. Unless specified otherwise, subsequent calculations will be carried out on the set $\{x\in B_{R'};\,|\nabla u(x)|>0\}$, which allows us to  apply Lemma \ref{lem:L(u-alpah-w-gamma)}.

\subsection{Estimates of $|\nabla \ln u|$}

To proceed, we shall first establish the estimates of $\mathcal L(w^\gamma\eta):=\mathcal L_0(w^\gamma\eta)$ in a specific local orthonormal frame.

\begin{lemma}\label{lem:e1}
Assume that $u$ is a positive solution of \eqref{eq1} in $\Omega$. Then for any $\gamma\geq1$, there exists a constant $C=C(\gamma)>0$ such that
\begin{align*}
  &w^{1-\gamma}\mathcal L(w^\gamma\eta)\\
  \leq&2\gamma\left[-(\gamma-1)+\frac{1-N+2\gamma^{1/2}}{N-1}\right]v_{11}^2\eta\\
  &-2\gamma \left[\left(q-1-\frac2{N-1}\right)uf''-p(f')^2\right]u^{p-1}|f'|^{q-2}w^{\frac{q+2}2}\eta\\
  &+2\gamma\left(\frac{f''}{f'}\right)'w^2\eta-\frac{2\gamma(1-\gamma^{-1/2})}{N-1}(f')^{2(q-1)}u^{2p}w^q\eta-\frac{2\gamma(1-\gamma^{-1/2})}{N-1}\left(\frac{f''}{f'}\right)^2w^2\eta\\
  &+w\left(|q|u^{p}|f'|^{q-1}w^{\frac{q-1}2}+2\left|\frac{f''}{f'}\right|w^{\frac12}\right)|\nabla\eta|+Cw\eta^{-1}|\nabla\eta|^2 +\sqrt{N}w|D^2\eta|
\end{align*}
holds in $\{x\in B_{R'};\,|\nabla u(x)|>0\}$.
\end{lemma}

\begin{proof}
 Set $\alpha=0$ in Lemma \ref{lem:L(u-alpah-w-gamma)}. Then for all $\gamma\geq 1$, we have
\begin{align}\label{Lalpha=0}\nonumber
&w^{1-\gamma}\mathcal L(w^\gamma\eta)\\\nonumber
\leq & \gamma \bigg[-(\gamma-1)w^{-1}|\nabla w|^2-2u^{p-1}\Big((q-1)u|f'|^{q-2}f''-p|f'|^q\Big)w^{\frac{q+2}2}+2\left(\frac{f''}{f'}\right)'w^2\\
&-2|D^2 v|^2\bigg]\eta+w\left(|q|u^{p}|f'|^{q-1}w^{\frac{q-1}2}+2\left|\frac{f''}{f'}\right|w^{\frac12}\right)|\nabla\eta|-2\gamma \langle\nabla w,\nabla \eta\rangle-w\Delta\eta.
\end{align}
Let $\{e_1,e_2,\ldots,e_N\}$ be an orthonormal frame of the tangent bundle $T\mathbb R^N$ on a domain with $w\neq 0$ such that $e_1=\nabla v/|\nabla v|$. Set $v_i=\partial v/\partial{e_i}$ for $i=\{1,2,\ldots,N\}$. Then a direct calculation shows that
\begin{equation*}
 v_1=\langle\nabla v,e_1\rangle=w^{\frac12},
\end{equation*}
and
\begin{equation}\label{v_11}
  v_{11}=\frac12 w^{-\frac12}w_1=\frac12w^{-1}\langle\nabla w,\nabla v\rangle.
\end{equation}
It follows from \eqref{v_11} and Cauchy--Schwarz's inequality that for $\gamma\geq 1$,
\begin{equation}\label{-4v_11}
 -(\gamma-1) w^{-1}|\nabla w|^2\leq -(\gamma-1)w^{-2}\langle\nabla w,\nabla v\rangle^2=-4(\gamma-1)v_{11}^2.
\end{equation}
By Cauchy's inequality, we arrive at
\begin{equation}\label{D2v-1}
  -|D^2v|^2\leq-v_{11}^2-\sum_{i=2}^N v_{ii}^2\leq -v_{11}^2-\frac{1}{N-1}\left(\sum_{i=2}^Nv_{ii}\right)^2.
\end{equation}
We begin by analyzing the expression for
\begin{equation*}
  \left(\sum_{i=2}^Nv_{ii}\right)^2=\left(-\Delta v+v_{11}\right)^2.
\end{equation*}
It follows from \eqref{eq-v} that
\begin{align}\label{sum-vii}\nonumber
&-\left(\sum_{i=2}^N v_{ii}\right)^2\\
=&-\left(-\frac{|f'|^q}{f'}u^pw^{\frac{q}2}+\frac{f''}{f'}w+v_{11}\right)^2\\\nonumber
=&-v_{11}^2-(f')^{2(q-1)}u^{2p}w^q-\left(\frac{f''}{f'}\right)^2w^2+2u^p|f'|^{q-2}f''w^{\frac{q+2}2}+2\frac{|f'|^q}{f'}u^pw^{\frac{q}2}v_{11}-2\frac{f''}{f'}wv_{11}.
\end{align}
Applying Young's inequality, we derive
\begin{equation}\label{2ab-1}
 2\frac{|f'|^q}{f'}u^pw^{\frac{q}2}v_{11}\leq \gamma^{-\frac12}|f'|^{2(q-1)}u^{2p}w^{q}+\gamma^{\frac12}v_{11}^2,
\end{equation}
and
\begin{equation}\label{2ab-2}
-2\frac{f''}{f'}wv_{11}\leq \gamma^{-\frac12}\left(\frac{f''}{f'}\right)^2w^2+\gamma^{\frac12}v_{11}^2.
\end{equation}
Substituting \eqref{sum-vii}--\eqref{2ab-2} into \eqref{D2v-1}, it  yields that
\begin{align}\label{D^2v<}\nonumber
  -|D^2v|^2\leq &\left(\frac{2\gamma^{1/2}}{N-1}-1\right)v_{11}^2-\frac{1-\gamma^{-1/2}}{N-1}|f'|^{2(q-1)}u^{2p}w^q\\
  &-\frac{1-\gamma^{-1/2}}{N-1}\left(\frac{f''}{f'}\right)^2w^2+\frac{2}{N-1}u^p|f'|^{q-2}f''w^{\frac{q+2}2}.
\end{align}

Next, we turn to terms containing $\nabla \eta$ and $\Delta\eta$. By virtue of Cauchy--Schwarz's and Young's inequalities, we arrive at
\begin{align}
-2\gamma \langle\nabla w,\nabla\eta\rangle\leq&\frac{\gamma(\gamma-1)}4 w^{-1}|\nabla w|^2\eta+C(\gamma)w\eta^{-1}|\nabla\eta|^2.
\end{align}
In addition, we have
\begin{equation}\label{Delta-eta}
-w\Delta\eta\leq \sqrt{N} w|D^2\eta|.
\end{equation}
From \eqref{-4v_11}, \eqref{D^2v<}--\eqref{Delta-eta} and \eqref{Lalpha=0}, we conclude that
\begin{align*}
  &w^{1-\gamma}\mathcal L(w^\gamma\eta)\\
  \leq&2\gamma\left[-(\gamma-1)+\frac{1-N+2\gamma^{1/2}}{N-1}\right]v_{11}^2\eta\\
  &-2\gamma \left[\left(q-1-\frac2{N-1}\right)uf''-p(f')^2\right]u^{p-1}|f'|^{q-2}w^{\frac{q+2}2}\eta\\
  &+2\gamma\left(\frac{f''}{f'}\right)'w^2\eta-\frac{2\gamma(1-\gamma^{-1/2})}{N-1}(f')^{2(q-1)}u^{2p}w^q\eta-\frac{2\gamma(1-\gamma^{-1/2})}{N-1}\left(\frac{f''}{f'}\right)^2w^2\eta\\
  &+w\left(|q|u^{p}|f'|^{q-1}w^{\frac{q-1}2}+2\left|\frac{f''}{f'}\right|w^{\frac12}\right)|\nabla\eta|+C(\gamma)w\eta^{-1}|\nabla\eta|^2 +\sqrt{N}w|D^2\eta|.
\end{align*}
This completes the proof of Lemma \ref{lem:e1}.
\end{proof}

\noindent\textbf{The proof of Theorem \ref{them:p+q<(N+3)/}.} We proceed with the proof in the following three steps.

{\bf Step 1.} {\it Construct a differential inequality via an appropriate auxiliary function.} Let
\begin{equation}\label{f=e^s}
  f(s)=-e^s.
\end{equation}
It follows immediately that  $f',f''<0$. With this choice, we can deduce that
\begin{equation}\label{f''/f'=1}
  uf''=-(f')^2,\quad \frac{f''}{f'}=1,\quad\text{and} \quad \left(\frac{f''}{f'}\right)'=0.
\end{equation}
Moreover, it follows from Lemma \ref{lem:e1} that
\begin{align}\label{L(w-gamma)*}\nonumber
  &w^{1-\gamma}\mathcal L(w^\gamma\eta)\\
  \leq&\gamma\left[-2(\gamma-1)+\frac{2\left(1-N+2\gamma^{1/2}\right)}{N-1}\right]v_{11}^2\eta+2\gamma \left(p+q-\frac{N+1}{N-1}\right)u^{p+q-1}w^{\frac{q+2}2}\eta\\\nonumber
  &-\frac{2\gamma\left[1-\gamma^{-1/2}-(2\gamma)^{-1}(N-1)\right]}{N-1}\left(u^{2(p+q-1)}w^q\eta+w^2\eta\right)-u^{2(p+q-1)}w^q\eta-w^2\eta\\\nonumber
  &+w\left(|q|u^{p+q-1}w^{\frac{q-1}2}+2w^{\frac12}\right)|\nabla\eta|+C(\gamma)w\eta^{-1}|\nabla\eta|^2+\sqrt{N}w|D^2\eta|.
\end{align}

{\bf Step 2.} {\it The sign for the coefficient of  $u^{p+q-1}w^{(q+2)/2}$.} Using the basic inequality, we obtain
\begin{align*}
&-\frac{2\gamma\left[1-\gamma^{-1/2}-(2\gamma)^{-1}(N-1)\right]}{N-1}\left(u^{2(p+q-1)}w^q\eta+w^2\eta\right)\\
\leq& -\frac{4\gamma}{N-1}\left[1-\gamma^{-1/2}-(2\gamma)^{-1}(N-1)\right]u^{p+q-1}w^{\frac{q+2}2}\eta,
\end{align*}
which together with \eqref{L(w-gamma)*} yields that
\begin{align}\label{O(gam)}\nonumber
  &w^{1-\gamma}\mathcal L(w^\gamma\eta)\\
  \leq&\gamma\left[-2(\gamma-1)+\frac{2\left(1-N+2\gamma^{1/2}\right)}{N-1}\right]v_{11}^2\eta-u^{2(p+q-1)}w^q\eta-w^2\eta\\\nonumber
  &+2\gamma \left[p+q-\frac{N+1}{N-1}-\frac{2\left(1-\gamma^{-1/2}-(2\gamma)^{-1}(N-1)\right)}{N-1}\right]u^{p+q-1}w^{\frac{q+2}2}\eta\\\nonumber
  &+w\left(|q|u^{p+q-1}w^{\frac{q-1}2}+2w^{\frac12}\right)|\nabla\eta|+C(\gamma)w\eta^{-1}|\nabla\eta|^2+\sqrt{N}w|D^2\eta|.
\end{align}
Observe that
\begin{equation*}
  1-\gamma^{-1/2}-(2\gamma)^{-1}(N-1)=1-\gamma^{-\frac12}+o\left(\gamma^{-\frac12}\right)\quad\text{as}\quad \gamma\to\infty,
\end{equation*}
then there exists a sufficiently large $\gamma_1=\gamma(N,p,q)>0$ such that for $\gamma\geq \gamma_1$,
\begin{align}\label{<1/2<0}\nonumber
 &p+q-\frac{N+1}{N-1}-\frac{2\left(1-\gamma^{-1/2}-(2\gamma)^{-1}(N-1)\right)}{N-1}\\
 =&p+q-\frac{N+3}{N-1}+O\left(\gamma^{-\frac12}\right)<0,
\end{align}
where we have employed the condition $p+q<(N+3)/(N-1)$. In addition,  there exists $\gamma_2=\gamma(N)>0$ such that for  $\gamma\geq \gamma_2$, the coefficient of $v_{11}^2\eta$ in \eqref{O(gam)} is negative. From now on, we fix
\begin{equation*}
  \gamma=\gamma_0(N,p,q)=\max\{\gamma_1,\gamma_2\}.
\end{equation*}
Combining estimates \eqref{O(gam)} and \eqref{<1/2<0}, we consequently deduce that
\begin{align}\label{-w2eta}\nonumber
  w^{1-\gamma}\mathcal L(w^\gamma\eta)\leq&-u^{2(p+q-1)}w^q\eta-w^2\eta+w\left(|q|u^{p+q-1}w^{\frac{q-1}2}+2w^{\frac12}\right)|\nabla\eta|\\
  &+C(N,p,q)w\eta^{-1}|\nabla\eta|^2+\sqrt{N}w|D^2\eta|.
\end{align}

{\bf Step 3.} {\it Maximum principle argument and the completion of proof.}  Set $z=w^\gamma\eta$. We reformulate differential inequality \eqref{-w2eta} using the variable $z$.  Applying Young's inequality yields
\begin{equation*}
  |q|u^{p+q-1}w^{\frac{q+1}2}|\nabla\eta|\leq \frac12 u^{2(p+q-1)}w^{q}\eta+C(q)w\eta^{-1}|\nabla\eta|^2.
\end{equation*}
Accordingly, \eqref{-w2eta} is rewritten as
\begin{align}\label{L(z)}\nonumber
  &\mathcal L(z)\\
  \leq&-w^{\gamma+1}\eta+2w^{\gamma+\frac12}|\nabla\eta|+Cw^{\gamma}\eta^{-1}|\nabla\eta|^2+\sqrt{N}w^{\gamma}|D^2\eta|\\\nonumber
=&z^{\frac{\gamma+1}{\gamma}}\eta^{-\frac{1}{\gamma}}\left(-1+2z^{-\frac{1}{2\gamma}}\eta^{-\frac{2\gamma-1}{2\gamma}}|\nabla\eta|+Cz^{-\frac1{\gamma}}\eta^{-\frac{2\gamma-1}{\gamma}}|\nabla\eta|^2+\sqrt{N}z^{-\frac{1}{\gamma}}\eta^{-\frac{\gamma-1}{\gamma}}|D^2\eta|\right),
\end{align}
where $C=C(N,p,q)>0$.

We now consider terms containing $|\nabla\eta|$ and $|D^2\eta|$. Observe  that for $\gamma>1$,
\begin{equation*}
  \frac{\gamma-1}{\gamma}<\frac{2\gamma-1}{2\gamma}<1.
\end{equation*}
Using the properties of $\eta$ in \eqref{dfi-eta} with $a=(2\gamma-1)/(2\gamma)$, it follows from \eqref{L(z)} that
\begin{align*}
  \mathcal L(z)\leq -z^{\frac{\gamma+1}{\gamma}}\eta^{-\frac{1}{\gamma}}\left(1-Cz^{-\frac{1}{2\gamma}}R^{-1}-Cz^{-\frac1{\gamma}}R^{-2}\right),
\end{align*}
where $C=C(N,p,q)>0$. Then there exists $C_1=C(N,p,q)>0$ such that for
\begin{equation*}
z>A=A(N,p,q,R):=C_1R^{-2\gamma},
\end{equation*}
it holds
\begin{equation*}
  Cz^{-\frac{1}{2\gamma}}R^{-1}+Cz^{-\frac1{\gamma}}R^{-2}<\frac12.
\end{equation*}
One can further deduce that
\begin{equation*}
  \mathcal L(z)\leq -\frac12z^{\frac{\gamma+1}{\gamma}}\eta^{-\frac{1}{\gamma}}<0\quad\text{in}\ D_A^+:= \{x\in B_{R'};\,z(x)>A\}.
\end{equation*}

Applying the regularity theory on elliptic equation \eqref{eq1}, we know $u$ is smooth on $D_A^+$. Since $\eta\leq1$ and $|\nabla u|=|f'||\nabla v|=|f'|w^{1/2}$, it follows from \eqref{f=e^s} that there exists $C=C(N,p,q,R)>0$ such that
\begin{equation}\label{nabla-u>C}
  |\nabla u|=uw^{\frac12}\geq C_1^{1/\gamma}uR^{-1}\geq C\quad \text{in}\ D_A^+.
\end{equation}
Therefore, if $q\leq 1$, \eqref{H_alpha} together with \eqref{nabla-u>C} yields that
\begin{align*}
|\mathcal H_0|\leq &|q|u^p|f'|^{q-1}w^{\frac{q-2}2}|\nabla v|+2\left|\frac{f''}{f'}\right||\nabla v|\\
=&|q|u^p |\nabla u|^{q-1}+2\left|\frac{f''}{f'}\right||\nabla v|\leq C(N,p,q,R)\quad\text{in}\ D_A^+.
\end{align*}
However, for $q>1$, the continuity of $u$ and $\nabla u$ guarantees the identical conclusion. Thus we achieve  $z\in C^2(D_A^+)$ and $\mathcal H_0\in L^\infty(D_A^+)$. It follows from Lemma \ref{lem:maxi-pric} that
\begin{equation*}
  z\leq C_1R^{-2\gamma}\quad\text{in}\ B_{R'}.
\end{equation*}
Using $z=w^\gamma\eta$ and $\eta=1$ in $B_{R/2}$, the preceding inequality implies that
\begin{equation*}
  \frac{|\nabla u|}{u}=\frac{|\nabla u|}{|f'|}=|\nabla v|=w^{\frac12}\leq CR^{-1}\quad \text{in}\ B_{R/2},
\end{equation*}
where $C=C(N,p,q)>0$. In particular, the estimates above remains valid at $x_0$. Owing to the arbitrariness of $x_0\in \Omega$, we conclude that
\begin{equation*}
  \left(\frac{|\nabla u|}{u}\right)(x)\leq C{\rm dist}^{-1}(x,\partial\Omega),\quad x\in\Omega,
\end{equation*}
where $C=C(N,p,q)>0$. The proof of Theorem \ref{them:p+q<(N+3)/} is complete.
\hfill$\Box$

\noindent\textbf{The proof of Corollary \ref{corol:p+q<}.}
Assume that $u$ is a positive solution of \eqref{eq1} in $\mathbb R^N$. Letting ${\rm dist}(x,\partial\Omega)\to \infty$ in Theorem \ref{them:p+q<(N+3)/}, we obtain that
\begin{equation*}
  \nabla u\equiv 0\quad\text{in}\ \mathbb R^N.
\end{equation*}
Then all positive solutions of \eqref{eq1} in $\mathbb R^N$ are constants, which contradicts that $0=\int_{\mathbb R^N}\langle\nabla u,\nabla\psi\rangle=\int_{\mathbb R^N} u^p|\nabla u|^q\psi=\infty$ for $\psi\in C^{\infty}_0(\mathbb R^N)$ for $q<0$. We therefore  derive the nonexistence of nontrivial positive solution of \eqref{eq1} in $\mathbb R^N$.
\hfill$\Box$

\subsection{Estimates of $|\nabla u^{-\sigma}|$ with $\sigma>0$}

We now derive a differential inequality for $u^\alpha w$ with $\alpha\neq 0$, which can be seen as a perturbation of the $\mathcal L(w^\gamma)$ established in Lemma \ref{lem:e1} with $\gamma=1$.

\begin{lemma}\label{lem:L(uw)}
Assume that $u$ a positive solution of \eqref{eq1} in $\Omega$. Then for any $\alpha\in \mathbb R$ and $0<\varepsilon<1$, there exists $C(\varepsilon)>0$ such that
\begin{align*}
  &u^{-\alpha}\mathcal L_\alpha(u^\alpha w\eta)\\\nonumber
  \leq&-\left[2\Big(q-1-\frac{2-\varepsilon}{N}\Big)uf''-\left(2p-\alpha (q-1)\right)|f'|^2\right]u^{p-1}|f'|^{q-2}w^{\frac{q+2}2}\eta\\
  &-\frac{2-\varepsilon}N|f'|^{2(q-1)}u^{2p}w^q\eta+\alpha \bigg[(\alpha+1)\frac{(f')^2}{u^2}+\frac{2f''}{u}-\frac{2-\varepsilon}{\alpha N}\left(\frac{f''}{f'}\right)^2+\frac2{\alpha}\left(\frac{f''}{f'}\right)'\bigg]w^2\eta\\
  &+w\left(|q|u^{p}|f'|^{q-1}w^{\frac{q-1}2}+2\left|\frac{f''}{f'}\right|w^{\frac12}\right)|\nabla\eta|-w\Delta\eta-2\langle\nabla w,\nabla \eta\rangle+C(\varepsilon)w\eta^{-1}|\nabla\eta|^2
\end{align*}
holds in $\{x\in B_{R'};\,|\nabla u(x)|>0\}$.
\end{lemma}

\begin{proof}
Taking $\gamma=1$ in Lemma \ref{lem:L(u-alpah-w-gamma)} gives that
\begin{align}\label{L-gamm=1}
  &u^{1-\alpha}\mathcal L_\alpha(u^\alpha w\eta)\\\nonumber
  \leq& u\left[-2|D^2 v|^2-\Big(2(q-1)u^p|f'|^{q-2}f''-\left(2p-\alpha (q-1)\right)u^{p-1}|f'|^q\Big)w^{\frac{q+2}2}+2\left(\frac{f''}{f'}\right)'w^2\right]\eta\\\nonumber
  &+\alpha \left((\alpha+1)\frac{(f')^2}{u}w+2f''\right)w^2\eta+uw\left(|q|u^{p}|f'|^{q-1}w^{\frac{q-1}2}+2\left|\frac{f''}{f'}\right|w^{\frac12}\right)|\nabla\eta|\\\nonumber
  &-uw\Delta\eta-2u\langle\nabla w,\nabla \eta\rangle.
\end{align}
By Cauchy--Schwarz's and Young's inequalities, for any $\varepsilon>0$, there exists $C(\varepsilon)>0$ such that
\begin{equation}\label{nablaw.nablaeta}
  2u\langle\nabla w,\nabla\eta\rangle\leq 2u|\nabla w||\nabla\eta|\leq 4u|D^2v|w^{\frac12}|\nabla\eta|\leq u\left(\varepsilon|D^2v|^2\eta+C(\varepsilon)w\eta^{-1}|\nabla \eta|^2\right),
\end{equation}
and together with \eqref{eq-v}, it yields that
\begin{align}\label{-D2v<-1/N*}\nonumber
  -|D^2 v|^2\eta\leq& -\frac1N(\Delta v)^2\eta\\
  =&-\frac1N\bigg[|f'|^{2(q-1)}u^{2p}w^q+\left(\frac{f''}{f'}\right)^2w^2-2u^p|f'|^{q-2}f''w^{\frac{q+2}2}\bigg]\eta.
\end{align}
 Substituting \eqref{nablaw.nablaeta} and \eqref{-D2v<-1/N*} into \eqref{L-gamm=1} leads to
\begin{align*}
  &u^{-\alpha}\mathcal L_\alpha(u^\alpha w\eta)\\\nonumber
  \leq&-\left(2\Big(q-1-\frac{2-\varepsilon}{N}\Big)uf''-\left(2p-\alpha (q-1)\right)|f'|^2\right)u^{p-1}|f'|^{q-2}w^{\frac{q+2}2}\eta\\
  &-\frac{2-\varepsilon}N|f'|^{2(q-1)}u^{2p}w^q\eta+\alpha \bigg[(\alpha+1)\frac{(f')^2}{u^2}+\frac{2f''}{u}-\frac{2-\varepsilon}{\alpha N}\left(\frac{f''}{f'}\right)^2+\frac2{\alpha}\left(\frac{f''}{f'}\right)'\bigg]w^2\eta\\
  &+w\left(|q|u^{p}|f'|^{q-1}w^{\frac{q-1}2}+2\left|\frac{f''}{f'}\right|w^{\frac12}\right)|\nabla\eta|-w\Delta\eta+C(\varepsilon)w\eta^{-1}|\nabla\eta|^2.
\end{align*}
Thus, we complete the proof of Lemma \ref{lem:L(uw)}.
\end{proof}

\noindent\textbf{The proof of Theorem \ref{them:p+q>=(N+3)/}.}
We take the auxiliary function $f$ as \eqref{f=e^s}. Since  $u\geq m>0$, the function $f$ maps $[\ln m,+\infty)$ into $(-\infty,-m]$. Combining Lemma \ref{lem:L(uw)} with  \eqref{f''/f'=1}, we deduce that
\begin{align}\label{alpha-espsilon}
&u^{-\alpha}\mathcal L_\alpha(u^\alpha w\eta) \\\nonumber
\leq& 2\left(p+\left(1-\frac{\alpha}2\right)q-\frac{N+2-\varepsilon}N+\frac{\alpha}2\right)u^{p+q-1}w^{\frac{q+2}2}\eta+\left(\alpha(\alpha-1)-\frac{2-\varepsilon}{N}\right)w^2\eta\\\nonumber
  & -\frac{2-\varepsilon}Nu^{2(p+q-1)}w^q\eta+|q|u^{p+q-1}w^{\frac{q+1}2}|\nabla\eta|+2w^{\frac32}|\nabla \eta|-w\Delta\eta+C(\varepsilon)w\eta^{-1}|\nabla\eta|^2.
\end{align}

Now we analyze the sign of coefficients of terms $u^{p+q-1}w^{(q+2)/2}\eta$ and $w^2\eta$. Let
\begin{equation*}
 Q(\alpha)= \alpha(\alpha-1)-\frac{2}{N}.
\end{equation*}
Equation $Q(\alpha)=0$ admits a real negative root
\begin{equation*}
  \alpha_0=\frac12-\sqrt{\frac{N+8}{4N}}<0.
\end{equation*}
Therefore, for any $\alpha_0<\alpha\leq0$, we choose $\varepsilon=\varepsilon(\alpha)>0$ small enough such that
\begin{equation}\label{alpha-0}
  \alpha(\alpha-1)-\frac{2-\varepsilon}{N}<\frac12Q(\alpha)<0.
\end{equation}
Furthermore, from the assumption \eqref{Q(p,q)<sigma} with $\sigma=-\alpha/2$, there exists $\varepsilon_0=\varepsilon(N,p,q,\alpha)>0$, such that for any $0<\varepsilon<\varepsilon_0$,
\begin{equation}\label{alpha-0<alpha}
p+\left(1-\frac{\alpha}2\right)q-\frac{N+2-\varepsilon}N+\frac{\alpha}2<0.
\end{equation}
By Young's inequality, we have
\begin{equation*}
  |q|u^{p+q-1}w^{\frac{q+1}2}|\nabla\eta|\leq \frac1N u^{2(p+q-1)}w^{q}\eta+C(N,q)w\eta^{-1}|\nabla\eta|^2.
\end{equation*}
Together \eqref{Delta-eta} with \eqref{alpha-espsilon}--\eqref{alpha-0<alpha}, we obtain that there exists $C=C(N,p,q,\alpha)>0$ such that
\begin{align*}
  &u^{-\alpha}\mathcal L_\alpha(u^\alpha w\eta)\\
  \leq& \frac12Q(\alpha)w^2\eta +2w^{\frac32}|\nabla \eta|+\sqrt{N}w|D^2\eta|+Cw\eta^{-1}|\nabla\eta|^2\\
=&w^2\eta\left(\frac12Q(\alpha)+2(w^2\eta^2)^{-\frac14}\eta^{-\frac12}|\nabla\eta|+\sqrt N(w^2\eta^2)^{-\frac12}|D^2\eta|+C(w^2\eta^2)^{-\frac12}\eta^{-1}|\nabla\eta|^2\right).
\end{align*}

Combining $Q(\alpha)<0$  with the properties of $\eta$ in \eqref{dfi-eta} for $a=1/2$, there exists $C_1=C(N,p,q,\alpha)>0$ such that when
\begin{equation*}
w\eta\geq C_1R^{-2},
\end{equation*}
we have
\begin{equation*}
2(w^2\eta^2)^{-\frac14}\eta^{-\frac12}|\nabla\eta|+\sqrt N(w^2\eta^2)^{-\frac12}|D^2\eta|+C(w^2\eta^2)^{-\frac12}\eta^{-1}|\nabla\eta|^2\leq-\frac14Q(\alpha).
\end{equation*}
Setting $z=u^{\alpha}w\eta$, from $u\geq m>0$, we further deduce that
\begin{equation*}
  \mathcal L_\alpha(z)\leq \frac14 Q(\alpha)u^\alpha w^2\eta<0\quad\text{in}\ D^+_A:=\{x\in B_{R'};\,z(x)>C_1m^{\alpha}R^{-2}\}.
\end{equation*}
From the regularity theory on elliptic equation \eqref{eq1}, we know that $u$ is smooth on $D_A^+$. Noting that $\eta\leq1$, and $|\nabla u|=|f'|w^{1/2}$, we obtain
\begin{equation}\label{nabla-low}
  |\nabla u|=uw^{\frac12}\geq u^{\frac{2-\alpha}2}z^{\frac12}\geq C(N,p,q,\alpha) mR^{-1}\quad\text{in}\ D_A^+.
\end{equation}
For $q<1$, it follows from \eqref{H_alpha} and \eqref{nabla-low} that
\begin{align*}
|\mathcal H_\alpha|\leq &|q|u^p|f'|^{q-1}w^{\frac{q-2}2}|\nabla v|+2\left|\frac{f''}{f'}\right||\nabla v|+2|\alpha|\left|\frac{f'}{f}\right||\nabla v|\\
=&|q|u^p |\nabla u|^{q-1}+2\left|\frac{f''}{f'}\right||\nabla v|+2|\alpha|\left|\frac{f'}{f}\right||\nabla v|\leq C(N,p,q,\alpha,m,R)
\end{align*}
in $D^+_A$. However, if $q\geq1$, the same conclusion remains valid owing to the continuity of $u$ and $\nabla u$. Consequently, we have $z\in C^2(D_A^+)$ and $\mathcal H_\alpha\in L^\infty(D_A^+)$. From Lemma \ref{lem:maxi-pric}, it follows that
\begin{equation*}
  z\leq C_1m^\alpha R^{-2}\quad\text{in}\ B_{R'}.
\end{equation*}
Using $z=u^\alpha |f'|^{-2}|\nabla u|^2\eta=u^{\alpha-2}|\nabla u|^2\eta$ and $\eta=1$ in $B_{R/2}$, the inequality above gives that for $\alpha<0$,
\begin{equation}\label{est-nabl-u^a}
  |\nabla u^{\frac{\alpha}2}|=\frac{|\alpha|}2u^{\frac{\alpha-2}2}|\nabla u|\leq C(N,p,q,\alpha)m^{\frac{\alpha}2}R^{-1}\quad \text{in}\ B_{R/2}.
\end{equation}
Particularly, \eqref{est-nabl-u^a} remains valid at $x_0$. Following the arbitrariness of $x_0\in\Omega$ and $\sigma=-\alpha/2$, we obtain that for any $0<\sigma<\big(\sqrt{(N+8)/N}-1\big)/4$,
\begin{equation*}
  |\nabla u^{-\sigma}(x)|\leq C(N,p,q,\sigma)m^{-\sigma}{\rm dist}^{-1}(x,\partial\Omega),\quad x\in\Omega.
\end{equation*}
This completes the proof of Theorem \ref{them:p+q>=(N+3)/}.
\hfill$\Box$

To derive the gradient estimates for positive solutions to \eqref{eq1} in $\mathbb R^N$, we recall a well-known lemma from \cite[Lemma 2.3]{Serrin-Zhou-Acta}, which provides a lower bound for positive weak super-harmonic functions.

\begin{lemma}\label{lem:u>C}
Suppose $\big\{|x|>R_0>0\big\}\subset\Omega$. Let $u$ be a positive weak solution of the inequality
\begin{equation*}
  -\Delta_mu\geq 0,\quad x\in\Omega,
\end{equation*}
where $\Delta_mu={\rm div}\big (|\nabla u|^{m-2}\nabla u\big)$. Then there exists a constant $C=C(m,N,u,R_0)>0$ such that
\begin{equation*}
  u(x)\geq C|x|^{-\frac{N-m}{m-1}}
\end{equation*}
provided $N>m$, while
\begin{equation*}
  \liminf_{|x|\to\infty}u(x)>0
\end{equation*}
if $N\leq m$.
\end{lemma}

We note that the proof of Lemma \ref{lem:u>C} is based upon the fact that
\begin{equation*}
  v(x)=C|x|^{-\frac{N-m}{m-1}}
\end{equation*}
is a fundamental solution of $-\Delta_m v=0$, together with the comparison principle, where
\begin{equation}\label{C(1)}
  C=R_0^{(N-m)/(m-1)}\min_{|x|=2R_0}u(x)>0.
\end{equation}

\noindent\textbf{The proof of Theorem \ref{corol:<p+q<}.}
Fix $x_0\in \mathbb R^N$ and  take $R>|x_0|+1$. According to Lemma \ref{lem:u>C} with $m=2$, and from \eqref{C(1)}, we obtain
\begin{equation}\label{u>cx-e}
  u(x)\geq C|x|^{-(N-2)},\quad |x|>1
\end{equation}
for $N>2$, where $C=\min_{|x|=2}u>0$. While if $N=2$, it follows that $u\geq l:=\inf_{\mathbb R^2}u>0$.

Set
\begin{equation*}
  m_R=\inf_{x\in B_R(x_0)}u(x).
\end{equation*}
Select $R$ large enough, such that $2^{2-N}CR^{-(N-2)}<\min_{|x|\leq 1}u$ for $N>2$. It follows from \eqref{u>cx-e} that
\begin{equation*}
m_R\geq \inf_{x\in B_{2R}(0)}u(x)\geq 2^{2-N}CR^{-(N-2)}.
\end{equation*}
Applying Theorem \ref{them:p+q>=(N+3)/}, we deduce that
\begin{equation*}
  |\nabla u^{-\sigma}(x_0)|\leq Cm_R^{-\sigma}R^{-1}\leq CR^{(N-2)\sigma-1}
\end{equation*}
for $N>2$ with $C=C(N,p,q,\sigma,\min_{|x|=2}u)>0$, and
\begin{equation*}
  |\nabla u^{-\sigma}(x_0)|\leq C(N,p,q,\sigma)m_R^{-\sigma}R^{-1}\leq C(N,p,q,\sigma,l)R^{-1}
\end{equation*}
for $N=2$.
It is clear that $(N-2)\sigma<1$ when $0<\sigma<\big(\sqrt{(N+8)/N}-1\big)/4$. Letting $R\to\infty$, we then obtain that $\nabla u^{-\sigma}(x_0)=0$. Due to the arbitrariness of $x_0\in\Omega$, we know $\nabla u^{-\sigma}\equiv0$ in $\mathbb R^N$ and hence $u$ is a constant, which contradicts that $0=\int_{\mathbb R^N}\langle\nabla u,\nabla\psi\rangle=\int_{\mathbb R^N} u^p|\nabla u|^q\psi=\infty$ for $\psi\in C^{\infty}_0(\mathbb R^N)$ when $q<0$.
\hfill$\Box$

\subsection{Estimates of $|\nabla u^{\beta}|$ with $\beta>0$}

We now choose a different auxiliary function from the one used above to derive the gradient estimates.

\noindent\textbf{Proof of Theorem \ref{them:p<=0}.}
We use the change of variable  $u=v^b$, as in \cite{Bidaut-Veron-Adv.N-2021}. Namely, we take $f(s)=-s^b$ with
\begin{equation*}
b=\frac{q-1}{p+q-1}.
\end{equation*}
Note that $p\leq 0$, and thus $b\geq 1$. With this choice, it has
\begin{equation*}
2(q-1)u^p|f'|^{q-2}f''-2pu^{p-1}|f'|^q=2\left((q-1)uf''-p|f'|^2\right)u^{p-1}|f'|^{q-2}=0,
\end{equation*}
\begin{equation*}
  \frac{f''}{f'}=\frac{b-1}{v},\quad\left(\frac{f''}{f'}\right)'=-\frac{b-1}{v^2},
\end{equation*}
and
\begin{equation*}
  \frac{|f'|^q}{f'}u^p=-b^{q-1}.
\end{equation*}
Setting $\alpha=0$ and $\gamma=1$ in Lemma \ref{lem:L(u-alpah-w-gamma)}, it follows from  \eqref{-D2v<-1/N*}, \eqref{eq-v}, \eqref{Delta-eta} and \eqref{nablaw.nablaeta} that
\begin{align*}
\mathcal L(w\eta)
  \leq & -(b-1)\frac{w^2}{v^2}\eta-\frac1N\left(b^{q-1}w^{\frac{q}2}+(b-1)\frac{w}v\right)^2\eta\\
  &+C(N)w|D^2\eta|+C(p,q)w^{\frac{q+1}2}|\nabla\eta|+C(p,q)\frac{w^{3/2}}v|\nabla\eta|+Cw\eta^{-1}|\nabla\eta|^2\\
  \leq &-\frac{b^{2(q-1)}}Nw^q\eta +C(N,p,q)\left(w|D^2\eta|+w^{\frac{q+1}2}|\nabla\eta|+w\eta^{-1}|\nabla\eta|^2\right)
\end{align*}
holds in $\{x\in B_{R'};\, |\nabla u(x)|>0\}$. Setting $z=w\eta$, we rewrite the inequality above as
\begin{align}\label{L(z)-u^b-**}\nonumber
  &\mathcal L(z)\\
  <&w^q\eta\left[-\frac{b^{2(q-1)}}N+C(N,p,q)\left(w^{1-q}\eta^{-1}|D^2\eta|+w^{\frac{1-q}2}\eta^{-1}|\nabla\eta|+w^{1-q}\eta^{-2}|\nabla\eta|^2\right)\right]\\\nonumber
  =&z^q\eta^{1-q}\left[-\frac{b^{2(q-1)}}N+C(N,p,q)\left(z^{1-q}\eta^{q-2}|D^2\eta|+z^{\frac{1-q}2}\eta^{\frac{q-3}2}|\nabla\eta|+z^{1-q}\eta^{q-3}|\nabla\eta|^2\right)\right].
\end{align}
Since $q>1$, we select $a=\max\{(3-q)/2,0\}<1$ in \eqref{dfi-eta}, and hence there exists $C_1=C(N,p,q)>0$ such that when
\begin{equation*}
z>A=A(N,p,q,R):=C_1R^{-\frac2{q-1}},
\end{equation*}
it has
\begin{equation*}
  C(N,p,q)\left(z^{1-q}R^{-2}+z^{\frac{1-q}2}R^{-1}\right)<\frac{b^{2(q-1)}}N,
\end{equation*}
which together with \eqref{L(z)-u^b-**} leads to that
\begin{equation*}
  \mathcal L(z)<0\quad \text{in}\ D^+_A:=\{x\in B_{R'};\,z>A\}.
\end{equation*}
By arguments analogous to those in the proof of Theorem \ref{them:p+q<(N+3)/}, we deduce that $z\leq C_1 R^{-2/(q-1)}$ in $B_{R'}$. Combining $z=w\eta$ together with the definition of $\eta$, we arrive at
\begin{equation*}
  |\nabla u^{\frac{p+q-1}{q-1}}(x_0)|=|\nabla v(x_0)|\leq C(N,p,q){\rm dist}^{-\frac1{q-1}}(x_0,\partial\Omega).
\end{equation*}
Since $x_0\in\Omega$ is arbitrary, the estimates above holds throughout $\Omega$. Letting ${\rm dist}(x,\partial\Omega)\to\infty$ in the case $\Omega=\mathbb R^N$, we conclude that $u$ is a constant. This completes the proof.
\hfill$\Box$

\section{Universal estimates via Liouville-type theorems}\label{sect:univer-esti}

\noindent\textbf{The proof of Theorem \ref{them:uni-u}.}
We set
 \begin{equation*}
   \beta=\frac{p+q-1}{2p+q},
 \end{equation*}
and
\begin{equation*}
  M(u)=\left(u^p|\nabla u|^{q}\right)^\beta.
\end{equation*}
The desired  estimates \eqref{uniform-u} can be written equivalently as
\begin{equation*}
  M(x):=M(u(x))\leq C{\rm dist}^{-1}(x,\partial\Omega),\quad x\in\Omega,
\end{equation*}
where $C=C(N,p,q)>0$.

Assume estimates \eqref{uniform-u} fails. Then there exist sequences of domains $\Omega_k$, positive solutions $u_k$ solves \eqref{eq1} in $\Omega_k$, and points $x_k\in \Omega_k$ such that
\begin{equation*}
  M_k(x_k):=\left(u^p_k|\nabla u_k|^{q}\right)^\beta(x_k)> 2k{\rm dist}^{-1}(x_k,\partial\Omega_k),\quad k=1,2,\ldots,
\end{equation*}
which means that
\begin{equation}\label{rk<}
  M_k(x_k){\rm dist}(x_k,\partial\Omega_k)> 2k.
\end{equation}
Defining
\begin{equation*}
  r_k=2kM_k^{-1}(x_k),
\end{equation*}
it follows from \eqref{rk<} that
\begin{equation*}
r_k<{\rm dist}(x_k,\partial\Omega_k),
\end{equation*}
and therefore $B_{r_k}(x_k)\subset\Omega_k$. To start the rescaling procedure, we frist consider the function
\begin{equation*}
  S_k(x):=M_k(x){\rm dist}(x,\partial B_{r_k}(x_k))=M_k(x)\left(r_k-|x-x_k|\right),\quad x\in B_{r_k}(x_k).
\end{equation*}
From \eqref{beta/p>0}, we have  $q\beta \geq0$, and thus $S_k$ is continuous.  Moreover,
\begin{equation*}
S_k(x)=0<S_k(x_k)=M_k(x_k)r_k=2k,\quad x\in\partial B_{r_k}(x_k).
\end{equation*}
Then there exists $a_k\in B_{r_k}(x_k)$ such that
\begin{equation}\label{max-S}
  S_k(a_k)=\max_{x\in \overline{B}_{r_k}(x_k)}S_k(x)=\max_{x\in B_{r_k}(x_k)}S_k(x),
\end{equation}
which implies that
\begin{equation}\label{rk-lambdak-Mxk}
\frac{M_k(x_k)}{M_k(a_k)}\leq \frac{r_k-|a_k-x_k|}{r_k}\leq 1.
\end{equation}

We now start the rescaling procedure. Set
\begin{equation}\label{def:labda}
  \lambda_k=M_k^{-1}(a_k).
\end{equation}
It follows from \eqref{rk-lambdak-Mxk} and the definition of $r_k$ that
\begin{equation*}
  2k\lambda_k=r_k\lambda_k M_k(x_k)\leq r_k-|a_k-x_k|.
\end{equation*}
This gives that for any $x\in B_{k\lambda_k}(a_k)$,
\begin{equation*}
|x-x_k|\leq |x-a_k|+|a_k-x_k|\leq k\lambda_k+r_k-2k\lambda_k\leq r_k,
\end{equation*}
and hence $B_{k\lambda_k}(a_k)\subset B_{r_k}(x_k)$. Furthermore, we know that for $x\in B_{k\lambda_k}(a_k)$,
\begin{align*}
2\left(r_k-|x-x_k|\right)&\geq 2\left(r_k-|x-a_k|-|a_k-x_k|\right)\\
&\geq 2r_k-2k\lambda_k-|a_k-x_k|+2k\lambda_k-r_k\\
&=r_k-|a_k-x_k|,
\end{align*}
from which it follows that
\begin{equation}\label{<2}
  \frac{r_k-|a_k-x_k|}{r_k-|x-x_k|}\leq 2, \quad x\in B_{k\lambda_k}(a_k).
\end{equation}
Combining \eqref{max-S} with \eqref{<2} yields that
\begin{equation}\label{M(x)<2M(ak)}
  M_k(x)\leq 2M_k(a_k),\quad x\in B_{k\lambda_k}(a_k).
\end{equation}
We rescale $u_k$ by setting
\begin{equation*}
  v_k(y):=\lambda_k^{(2-q)/(p+q-1)}u_k(\lambda_ky+a_k),\quad y\in B_k(0).
\end{equation*}
Note from \eqref{eq1} that $v_k$ satisfies
\begin{equation*}
  -\Delta v_k =v_k^p|\nabla v_k|^q\quad \text{in}\ B_k(0).
\end{equation*}
We deduce from \eqref{def:labda} and \eqref{M(x)<2M(ak)} that
\begin{equation}\label{u-nablau(0)=1}
  \left(v^p_k|\nabla v_k|^{q}\right)^\beta(0)=\lambda_k M_k(a_k)=1,
\end{equation}
and
\begin{equation*}
  \left(v^p_k|\nabla v_k|^{q}\right)^\beta(y)=\lambda_kM_k(\lambda_ky+a_k)\leq 2,\quad y\in B_k(0).
\end{equation*}
Then for $\beta>0$, we conclude that
\begin{equation*}
  v_k^p|\nabla v_k|^q\leq 2^{\frac1{\beta}}\quad \text{in}\ B_k(0).
\end{equation*}
Using the regularity theory of elliptic equation,  we obtain that there exists $s\in (0,1)$ such that $v_k$ is bounded in $C^{2,s}_{\rm loc}(\mathbb R^N)$. Then, up to a subsequence, there exists a nonnegative function $v\in C^2_{loc}(\mathbb R^N)$ such that
\begin{equation*}
  v_k\to v\quad \text{in}\ C^2_{loc}(\mathbb R^N),\quad k\to \infty.
\end{equation*}
 Moreover, we obtain from \eqref{u-nablau(0)=1} that $v^p|\nabla v|^q(0)=1$,  and hence $v$ is nontrivial. Thus, by the strong maximum principle, $v$ is a positive nontrivial solution to equation \eqref{eq1} in $\mathbb R^N$. This contradicts the Liouville-type theorem for $(p,q)\in \mathcal R_L$ and thus Theorem \ref{them:uni-u} holds.
\hfill$\Box$

\vskip 3mm
\noindent{\bf Conflict of interest.} {No potential conflict of interest was reported by the authors.}

\vskip 3mm
\noindent{\bf Data availability.} {No data was used for the research described in the article.}

\vskip 3mm
\noindent{\bf Acknowledgments.} {The research of W.G. Liang was supported by the  Fundamental  Research  Funds  for  the Central Universities (No. xzy022025046). The research of Z.C. Zhang was partially supported by the National Natural
Science Foundation of China (Nos. 12271423 and 12671248).
}


\begin{thebibliography}{10}

\bibitem{Attouchi-CVPDE-2020}
A. Attouchi and Ph. Souplet,
\newblock Gradient blow-up rates and sharp gradient estimates for diffusive
  {H}amilton-{J}acobi equations,
\newblock {\em Calc. Var. Partial Differ. Equ.} 59 (2020), Paper No. 153, 28 pp.

\bibitem{Baldelli-Filippucci-2025-RIMU}
L.~Baldelli and R.~Filippucci,
\newblock A priori estimates for convective quasilinear equations and systems,
\newblock {\em Rend. Istit. Mat. Univ. Trieste.} 57 (2025), Paper No. 16, 28 pp.

\bibitem{Bidaut-Veron-Adv.N-2021}
M.-F. Bidaut-V\'{e}ron,
\newblock Liouville results and asymptotics of solutions of a quasilinear
  elliptic equation with supercritical source gradient term,
\newblock {\em Adv. Nonlinear Stud.} 21 (1) (2021) 57--76.

\bibitem{Veron-2014-JFA}
M.-F. Bidaut-V\'{e}ron, M. Garc\'{\i}a-Huidobro, and L.
  V\'{e}ron,
\newblock Local and global properties of solutions of quasilinear
  {H}amilton-{J}acobi equations,
\newblock {\em J. Funct. Anal.} 267 (9) (2014) 3294--3331.

\bibitem{Bidaut-Veron-Duke-2019}
M.-F. Bidaut-V\'{e}ron, M. Garc\'{\i}a-Huidobro, and L.
  V\'{e}ron,
\newblock Estimates of solutions of elliptic equations with a source reaction
  term involving the product of the function and its gradient,
\newblock {\em Duke Math. J.} 168 (8) (2019) 1487--1537.

\bibitem{Caffarelli-Giads-CPAM-1989}
L.A. Caffarelli, B. Gidas, and J. Spruck,
\newblock Asymptotic symmetry and local behavior of semilinear elliptic
  equations with critical {S}obolev growth,
\newblock {\em Comm. Pure Appl. Math.} 42 (3) (1989) 271--297.

\bibitem{Caristi-AdcanceD-1997}
G. Caristi and E. Mitidieri,
\newblock Nonexistence of positive solutions of quasilinear equations,
\newblock {\em Adv. Differ. Equ.} 2 (3) (1997) 319--359.

\bibitem{Chang-Hu-Zhang-NA-2022}
C.H. Chang, B. Hu, and Z.C. Zhang,
\newblock Liouville-type theorems and existence of solutions for quasilinear
  elliptic equations with nonlinear gradient terms,
\newblock {\em Nonlinear Anal.} 220 (2022), Paper No. 112873, 29 pp.

\bibitem{Chang-JDE-2023}
C.H. Chang, B. Hu, and Z.C. Zhang,
\newblock Gradient blowup behavior for a viscous {H}amilton-{J}acobi equation
  with degenerate gradient nonlinearity,
\newblock {\em J. Differ. Equ.} 359 (2023) 23--66.

\bibitem{Chang-DCDS-2020}
C.H. Chang, Q.C. Ju, and Z.C. Zhang,
\newblock Asymptotic behavior of global solutions to a class of heat equations
  with gradient nonlinearity,
\newblock {\em Discrete Contin. Dyn. Syst.} 40 (10) (2020) 5991--6014.

\bibitem{Chang-Zhang-NA-2026}
C.H. Chang and Z.C. Zhang,
\newblock A priori estimates and existence of solutions for quasilinear
  elliptic equations with nonlinear gradient terms,
\newblock {\em Nonlinear Anal.} 267 (2026), Paper No. 114064, 18 pp.

\bibitem{Chen-Li-Duke-1991}
W.X. Chen and C.M. Li,
\newblock Classification of solutions of some nonlinear elliptic equations,
\newblock {\em Duke Math. J.} 63 (3) (1991) 615--622.

\bibitem{Chen-Li-Wu-Xin-MathAnn-2026}
W.X. Chen, C.M. Li, L.Y. Wu, and Z.P. Xin,
\newblock Refined regularity for nonlocal elliptic equations and applications,
\newblock {\em Math. Ann.} 395 (2026), Paper No. 60, 40 pp.

\bibitem{Cianchi-CPDE-2011}
A. Cianchi and V.G. Maz'ya,
\newblock Global {L}ipschitz regularity for a class of quasilinear elliptic
  equations,
\newblock {\em Comm. Partial Differ. Equ.} 36 (1) (2011) 100--133.

\bibitem{Girant-Goffi-ARMA-2021}
M.~Cirant and A.~Goffi,
\newblock On the problem of maximal {$L^q$}-regularity for viscous
  {H}amilton-{J}acobi equations,
\newblock {\em Arch. Ration. Mech. Anal.} 240 (3) (2021) 1521--1534.

\bibitem{Dancer-MathZ-1998}
E.N. Dancer,
\newblock Superlinear problems on domains with holes of asymptotic shape and
  exterior problems,
\newblock {\em Math. Z.} 229 (3) (1998) 475--491.

\bibitem{Filipuucci-NA-2009}
R. Filippucci,
\newblock Nonexistence of positive weak solutions of elliptic inequalities,
\newblock {\em Nonlinear Anal.} 70 (8) (2009) 2903--2916.

\bibitem{Fillippucci-JDE-2011}
R. Filippucci,
\newblock Nonexistence of nonnegative solutions of elliptic systems of
  divergence type,
\newblock {\em J. Differ. Equ.} 250 (1) (2011) 572--595.

\bibitem{Filippucci-Pucci-Souplet-Adv.stu-2020}
R. Filippucci, P. Pucci, and Ph. Souplet,
\newblock A {L}iouville-type theorem for an elliptic equation with
  superquadratic growth in the gradient,
\newblock {\em Adv. Nonlinear Stud.} 20 (2) (2020) 245--251.

\bibitem{Filippucci-Pucci-Souplet-2020-CPDE}
R. Filippucci, P. Pucci, and Ph. Souplet,
\newblock A {L}iouville-type theorem in a half-space and its applications to
  the gradient blow-up behavior for superquadratic diffusive
  {H}amilton-{J}acobi equations,
\newblock {\em Comm. Partial Differ. Equ.} 45 (4) (2020) 321--349.

\bibitem{Gidas-Spruck-1981-CPAM}
B.~Gidas and J.~Spruck,
\newblock Global and local behavior of positive solutions of nonlinear elliptic
  equations,
\newblock {\em Comm. Pure Appl. Math.} 34 (4) (1981) 525--598.

\bibitem{Grenon-CRMA-2006}
N. Grenon, F. Murat, and A. Porretta,
\newblock Existence and a priori estimate for elliptic problems with
  subquadratic gradient dependent terms,
\newblock {\em C. R. Math. Acad. Sci. Paris.} 342 (1) (2006) 23--28.

\bibitem{Guo-Zhang-ProAMS-2025}
C. Guo and Z.C. Zhang,
\newblock A {L}iouville theorem for the quasilinear elliptic inequality on
  complete {R}iemannian manifolds,
\newblock {\em Proc. Amer. Math. Soc.} 153 (3) (2025) 1069--1075.

\bibitem{Handong-Wang-2025-JDE}
D. Han, J. He, and Y.D. Wang,
\newblock Universal gradient estimates for quasilinear {H}amilton-{J}acobi type
  equation on manifolds and {L}iouville theorems,
\newblock {\em J. Differ. Equ.} 449 (2025), Paper No. 113724, 36 pp.

\bibitem{Han-He-Wang-JFA-2026}
D. Han, J. He, and Y.D. Wang,
\newblock Gradient estimates for {$\Delta_pu + A|\nabla u|^q + Bu^r + C = 0$}
  on manifolds and applications,
\newblock {\em J. Funct. Anal.} 290 (2026), Paper No. 111274, 40 pp.

\bibitem{He-Hu-Wang-MZ-2026}
J. He, J.C. Hu, and Y.D. Wang,
\newblock Nash--{M}oser iteration approach to the logarithmic gradient estimates
  and {L}iouville properties of quasilinear elliptic equations on manifolds,
\newblock {\em Math. Z.} 313 (2026), Paper No. 6, 41 pp.

\bibitem{Tommaso-Porretta-2016-CPDE}
T. Leonori and A. Porretta,
\newblock Large solutions and gradient bounds for quasilinear elliptic
  equations,
\newblock {\em Comm. Partial Differ. Equ.} 41 (6) (2016) 952--998.

\bibitem{Lions-1985-JAM}
P.-L. Lions,
\newblock Quelques remarques sur les probl\`emes elliptiques quasilin\'{e}aires
  du second ordre,
\newblock {\em J. Anal. Math.} 45 (1985) 234--254.

\bibitem{Lu-arXive-2026}
Z.H. Lu,
\newblock Global and local properties of solutions of elliptic equations with a
  nonlinear term involving the product of the function and its gradient, 
\newblock {arXiv:2602.19999,} preprint (2026).

\bibitem{Lu-Zhu-JFA-2026}
Z.H. Lu and L.L. Zhu,
\newblock Liouville theorems and gradient estimates for a class of quasilinear
  elliptic equations with product gradient structures,
\newblock {\em J. Funct. Anal.} 291 (2026), Paper No. 111620, 34 pp.

\bibitem{Ma-Wu-Bull-2026}
X.-N. Ma and W.Z. Wu,
\newblock Liouville theorem for elliptic equations with a source reaction term
  involving the product of the function and its gradient in {$\Bbb R^n$},
\newblock {\em Bull. Sci. Math.} 206 (2026), Paper No. 103747, 30 pp.

\bibitem{Pokhozhaev-2001}
\`E. Mitidieri and S.~I. Pokhozhaev,
\newblock A priori estimates and the absence of solutions of nonlinear partial
  differential equations and inequalities,
\newblock {\em Trudy Mat. Inst. Steklov.} 234 (2001) 1--384.

\bibitem{Polacik-Quittner-Souplet}
P. Pol\'{a}\v{c}ik, P. Quittner, and Ph. Souplet,
\newblock Singularity and decay estimates in superlinear problems via
  {L}iouville-type theorems. {I}. {E}lliptic equations and systems,
\newblock {\em Duke Math. J.} 139 (3) (2007) 555--579.

\bibitem{Q-W-J-CVPDE-2026}
Z. Qiu, Y.D. Wang, and J. Yang,
\newblock Gradient estimates for {$p$}-{L}aplacian equation with cubic
  polynomial nonlinearity on {R}iemannian manifolds,
\newblock {\em Calc. Var. Partial Differ. Equ.}  65 (8) (2026), Paper No. 246, 35 pp.

\bibitem{Serrin-Zhou-Acta}
J. Serrin and H.H. Zou,
\newblock Cauchy-{L}iouville and universal boundedness theorems for quasilinear
  elliptic equations and inequalities,
\newblock {\em Acta Math.} 189 (1) (2002) 79--142.

\bibitem{Souplet-Zhang}
Ph. Souplet and Q.S. Zhang,
\newblock Global solutions of inhomogeneous {H}amilton-{J}acobi equations,
\newblock {\em J. Anal. Math.} 99 (2006) 355--396.

\bibitem{Sun-Xiao-Xu-MathAnn-2022}
Y.H. Sun, J. Xiao, and F.H. Xu,
\newblock A sharp {L}iouville principle for {$\Delta_m u+u^p|\nabla u|^q\le0$}
  on geodesically complete noncompact {R}iemannian manifolds,
\newblock {\em Math. Ann.} 384 (3--4) (2022) 1309--1341.

\bibitem{Wang-Wei-JDE-2023}
Y.D. Wang and G.D. Wei,
\newblock On the nonexistence of positive solution to {$\Delta u + au^{p+1} =
  0$} on {R}iemannian manifolds,
\newblock {\em J. Differ. Equ.} 362 (2023) 74--87.



\end{thebibliography}
\end{document}